\documentclass[12pt]{amsart}
\usepackage[top=1in, bottom=1in, left=1in, right=1in]{geometry}
\usepackage{graphicx,setspace,appendix}
\usepackage{amsmath,amsthm,amscd,amssymb, color}
\usepackage{amssymb,mathrsfs}

\usepackage[normalem]{ulem}

\definecolor{pink}{rgb}{1,.2,.6}

\definecolor{orange}{rgb}{0.7,0.3,0}

\definecolor{blue}{rgb}{.2,.6,.75}

\definecolor{green}{rgb}{.4,.7,.4}

\newcommand{\kommentar}[1]

\newcommand{\Z}{\mathbb Z}
\newcommand{\N}{\mathbb N}
\newcommand{\Q}{\mathbb Q}
\newcommand{\F}{\mathbb F}
\newcommand{\C}{\mathbb C}

\renewcommand{\L}{\mathcal L}

\newcommand{\dee}{\mathrm{d}}

\newcommand{\J}{\mathcal J}

\newcommand{\G}{G_{m, k}}

\DeclareMathOperator{\Aut}{Aut}
\DeclareMathOperator{\GL}{GL}
\DeclareMathOperator{\SL}{SL}
\DeclareMathOperator{\tr}{tr}
\DeclareMathOperator{\Mat}{M}

\newcommand{\isom}{\cong}

\newcommand{\leg}[2]{\left(\frac{#1}{#2}\right)}

\newcommand{\eq}[2]{ \begin{equation} \label{#1}\begin{split} #2 \end{split} \end{equation} }

\newcommand{\als}[1]{\begin{align*} #1 \end{align*} }

\renewcommand{\pmod}[1]{\,(\mathrm{mod}\,#1)}
\renewcommand{\mod}[1]{\,(\mathrm{mod}\,#1)}

\newcommand{\nn}{\nonumber \\}

\newcommand{\bv}\boldsymbol{}

\newtheorem{thm}{Theorem}[section]
\newtheorem{cor}[thm]{Corollary}
\newtheorem{prop}[thm]{Proposition}
\newtheorem{lma}[thm]{Lemma}

\theoremstyle{remark}
\newtheorem{rmk}[thm]{Remark}

\numberwithin{equation}{section}

\title{The frequency of elliptic curve groups over prime finite fields}
\author{Vorrapan Chandee}
\address[Vorrapan Chandee]{Department of Mathematics \\ Burapha University \\ 169 Long-hard Bangsaen rd, Saen suk, Mueang, Chonburi, 20131 Thailand}
\email{{\tt vorrapan@buu.ac.th}}
\author{Chantal David}
\address[Chantal David]{
Department of Mathematics and Statistics\\
Concordia University\\
1455 de Maisonneuve West\\
Montr\'eal, Qu\'ebec\\
H3G 1M8\\
Canada
}
\email{cdavid@mathstat.concordia.ca}
\author{Dimitris Koukoulopoulos}
\address[Dimitris Koukoulopoulos]{D\'epartement de math\'ematiques et de statistique\\
Universit\'e de Montr\'eal\\
CP 6128 succ. Centre-Ville\\
Montr\'eal, QC H3C 3J7\\
Canada}
\email{{\tt koukoulo@dms.umontreal.ca}}
\author{Ethan Smith}
\address[Ethan Smith]{
Department of Mathematics\\
Liberty University\\
1971 University Blvd\\
MSC Box 710052\\
Lynchburg, VA  24502
}
\email{ecsmith13@liberty.edu}

\subjclass[2010]{11G07, 11N45 (primary) 11N13, 11N36 (secondary)   }
\date{March 4, 2015}
\begin{document}

\maketitle

\begin{abstract}
Letting $p$ vary over all primes and $E$ vary over all elliptic curves over the finite field $\F_p$, we study the frequency to which a given group $G$ arises as a group of points $E(\F_p)$.
It is well-known that the only permissible groups are of the form $G_{m,k}:=\Z/m\Z\times \Z/mk\Z$.
Given such a candidate group, we let $M(\G)$ be the frequency to which the group $\G$ arises in this way.
Previously, the second and fourth named authors determined an asymptotic formula for $M(\G)$ assuming a conjecture about primes in short arithmetic progressions. In this paper, we prove several unconditional bounds for $M(\G)$, pointwise and on average. In particular, we show that $M(\G)$ is bounded above by a constant multiple of the expected quantity when $m\le k^A$ and that the conjectured asymptotic for $M(\G)$ holds for almost all groups $\G$ when $m\le k^{1/4-\epsilon}$.
We also apply our methods to study the frequency to which a given integer $N$ arises as the group order $\#E(\F_p)$.
\end{abstract}

%%%%%%%%%%%%%%%%%%%%%%%%%%%%%%%%%%%%%%%%%%%%%%%%%%%%%%%%%%%%%%%%%%%%%%%%%%%%%%%%%%%%%%%%%%%%%%%%%%%%%%%%%%%%%%%%%%%%%%%%%%%%%%%%%%%%%%%%%%%%%%%%%%%%%%%%%%%%%%%%%%%%%%%%%%%%%%%%%%%%%%%%%%%%%%%%%%%%%%%%%%%%%%%%%%%%%%%%%%%%%%%%%%%%%%%%%%%%%%%%%%%%%%%%%%%%%%%%%%%%%%%%%%%%%%%%%%%%%%%%%%%%%%%%%%%%%%%%%%%%%%%%%%%%%%%%%%%%%%%%%%%%%%%%%%%%%%%%%%%%%%%%%%%%%%%%%%%%%%%%%%%%%%

\section{Introduction}\label{intro}

Given an elliptic curve $E$ over the prime finite field $\F_p$, we let $E(\F_p)$ denote its set of $\F_p$ points.
It is well-known that $E(\F_p)$ admits the structure of an abelian group, and in fact,
\[
E(\F_p)\isom  G_{m,k}:= \Z/m\Z\times\Z/mk\Z
\]
for some positive integers $m$ and $k$. It is natural to wonder which groups of the form $G_{m,k}$ arise in this way and how often they occur as $p$ varies over all primes and $E$ varies over all elliptic curves over $\F_p$.
The former problem, of characterizing which groups are realized in this way was studied in~\cite{BPS:2012, CDKS1}, while the frequency of occurrence was studied by the second and fourth named authors in~\cite{DS-MEG}.
In the present work, we explore the frequency of occurrence further.

Given a group $G$ of the form $G_{m,k}=\Z/m\Z\times\Z/mk\Z$, we set $N=|G|=m^2k$ and let $M_p(G)$ denote the weighted number of isomorphism classes of elliptic curves over $\F_p$ with group isomorphic to $G$, that is to say
\[
M_p(G)=\sum_{\substack{E/\F_p\\ E(\F_p)\isom G}}\frac{1}{|\Aut_p(E)|},
\]
where the sum is taken over all isomorphism classes of elliptic curves over $\F_p$ and
$|\Aut_p(E)|$ is the number of $\F_p$-automorphisms of $E$.
It is worth noting here that $|\Aut_p(E)|=2$ for all but a bounded number of isomorphism classes $E$ over $\F_p$, and hence
\[
M_p(G) = \frac{1}{2} \#\{ E/\F_p  :  E(\F_p)\isom G \} +O(1),
\]

In \cite{DS-MEG}, the authors studied the weighted number
of isomorphism classes of elliptic curves over any prime finite field with group of points isomorphic to $G$, i.e., they studied
\[
M(G):=\sum_pM_p(G) .
\]
The primes counted by $M(G)$ must lie in a very short interval near $N=|G|$.
This is because the Hasse bound implies that $p+1-2\sqrt p<N<p+1+2\sqrt p$, which is equivalent to saying that
\[
N^{-} := N+1-2\sqrt{N} < p < N+1+2 \sqrt{N}=:N^+.
\]
Even the Riemann hypothesis does not guarantee the existence of a prime in such a short interval.
Hence the main theorem of \cite{DS-MEG} can only be proven under an appropriate conjecture concerning the distribution of primes in short intervals.
In the statement below, we refer to the conjecture assumed in~\cite{DS-MEG} as the
Barban-Davenport-Halberstam (BDH) estimate for short intervals.

Before stating the main theorem of~\cite{DS-MEG}, we fix some more notation.
Given a group $G=G_{m,k}$, we let $\Aut(G)$ denote its automorphism group (as a group).
This should not be confused with $\Aut_p(E)$ as defined above,
which refers to the set of $\F_p$-automorphisms of the elliptic curve $E$.
We also define the function
\eq{define K(G)}{
K(G)=
\prod_{\ell\nmid N}\left(1-\frac{\leg{N-1}{\ell}^2\ell+1}{(\ell-1)^2(\ell+1)}\right)
\prod_{\ell\mid m}\left(1-\frac{1}{\ell^2}\right)
\prod_{\substack{\ell\mid k\\ \ell\nmid m}}\left(1-\frac{1}{\ell(\ell-1)}\right),
}
where the products are taken over all primes $\ell$ satisfying the stated conditions and $\leg{\cdot}{\ell}$ denotes the usual Kronecker symbol.
In~\cite{DS-MEG}, the function $K(G)$ was only computed for odd order groups, and its definition contained a mistake.
It was corrected to the form that we give here in~\cite{DS-MEG-corr}.
Note that the function $K(G)$ is bounded between two constants independently of the the parameters $m$ and $k$. In paraphrased form, the main theorem of~\cite{DS-MEG} is as follows.

%%%%%%%%%%%%%%%%%%%%%%%%%%%%%%%%%%%%%%%%%%%%%%%%%%%%%%%%%%%%%%%%%%%%%%%%%%%%%%%%%%%%%%%%%%%%%%%%%%%%%%%%%%%%%%%%%%%%%%%%%%%%%%%%%%%%%%%%%%%%%%%%%%%%%%%%%%%%%%%%%%%%%%%%%%%%%%%%%%%%%%%

\begin{thm}[David-Smith] \label{meg-rephrased}
Assume that the BDH estimate for short intervals holds.
Fix $A,B>0$.
Then for every nontrivial, odd order group $G=G_{m,k}$, we have that
\[
M(G)=\left(K(G) + O_{A,B}\left( \frac{1}{(\log|G|)^{B}} \right)\right)\frac{|G|^2}{|\Aut (G)|\log |G|}
	\asymp \frac{mk^2}{\phi(m)\phi(k)\log k},
\]
provided that $m\le (\log k)^A$.
\end{thm}

%%%%%%%%%%%%%%%%%%%%%%%%%%%%%%%%%%%%%%%%%%%%%%%%%%%%%%%%%%%%%%%%%%%%%%%%%%%%%%%%%%%%%%%%%%%%%%%%%%%%%%%%%%%%%%%%%%%%%%%%%%%%%%%%%%%%%%%%%%%%%%%%%%%%%%%%%%%%%%%%%%%%%%%%%%%%%%%%%%%%%%%

For precise details concerning the conjecture assumed to prove Theorem~\ref{meg-rephrased},
we refer the reader to \cite{DS-MEG}. We note that the result of Theorem~\ref{meg-rephrased} is restricted to the range $m\le (\log k)^A$. However, we believe that it should hold in the range $m\le k^A$.
Proving such a result at the present time would however require an even stronger hypothesis than the one assumed in \cite{DS-MEG}. Unconditionally, it is possible to obtain upper bounds of the correct order of magnitude in this larger range.
This is the context of our first theorem.

%%%%%%%%%%%%%%%%%%%%%%%%%%%%%%%%%%%%%%%%%%%%%%%%%%%%%%%%%%%%%%%%%%%%%%%%%%%%%%%%%%%%%%%%%%%%%%%%%%%%%%%%%%%%%%%%%%%%%%%%%%%%%%%%%%%%%%%%%%%%%%%%%%%%%%%%%%%%%%%%%%%%%%%%%%%%%%%%%%%%%%%

\begin{thm} \label{CDKS} Fix $A>0$ and consider integers $m$ and $k$ with $1\le m\le k^A$. Let $G=G_{m,k}$, $N=|G|=m^2k$, and
\[
\delta = \frac{1}{N/(\phi(m)\log(2N))}
	\sum_{\substack{ N^-<p\le N^+ \\  p\equiv 1\pmod m }} \sqrt{(p-N^-)(N^+-p)},
\]
and note that $\delta\ll1$ by the Brun-Titchmarsch inequality. For any fixed $\lambda>1$,
\[
\delta^\lambda \cdot \frac{|G|^2}{ |\Aut (G)|\log(2|G|)} \ll M(G)  \ll \delta^{1/\lambda} \cdot \frac{|G|^2}{ |\Aut (G)|\log(2|G|)},
\]
the implied constants depending at most on $A$ and $\lambda$.
\end{thm}

%%%%%%%%%%%%%%%%%%%%%%%%%%%%%%%%%%%%%%%%%%%%%%%%%%%%%%%%%%%%%%%%%%%%%%%%%%%%%%%%%%%%%%%%%%%%%%%%%%%%%%%%%%%%%%%%%%%%%%%%%%%%%%%%%%%%%%%%%%%%%%%%%%%%%%%%%%%%%%%%%%%%%%%%%%%%%%%%%%%%%%%

Employing the above result together with the Bombieri-Vinogradov theorem, we also show that the lower bound implicit in Theorem~\ref{meg-rephrased} holds for a positive proportion of groups $G$.

%%%%%%%%%%%%%%%%%%%%%%%%%%%%%%%%%%%%%%%%%%%%%%%%%%%%%%%%%%%%%%%%%%%%%%%%%%%%%%%%%%%%%%%%%%%%%%%%%%%%%%%%%%%%%%%%%%%%%%%%%%%%%%%%%%%%%%%%%%%%%%%%%%%%%%%%%%%%%%%%%%%%%%%%%%%%%%%%%%%%%%%

\begin{thm} \label{CDKS2} Consider numbers $x$ and $y$ with $1\le x\le \sqrt{y}$.
Then there are absolute positive constants $c_1$ and $c_2$ such that
\[
M(\G) \ge c_1\cdot \frac{|G_{m,k}|^2}{|\Aut (G_{m,k})|\log(2|G_{m,k}|)}
\]
for at least $c_2xy$ pairs $(m,k)$ with $m\le x$ and $k\le y$.
\end{thm}

%%%%%%%%%%%%%%%%%%%%%%%%%%%%%%%%%%%%%%%%%%%%%%%%%%%%%%%%%%%%%%%%%%%%%%%%%%%%%%%%%%%%%%%%%%%%%%%%%%%%%%%%%%%%%%%%%%%%%%%%%%%%%%%%%%%%%%%%%%%%%%%%%%%%%%%%%%%%%%%%%%%%%%%%%%%%%%%%%%%%%%%

\begin{rmk}
It is not possible for such a lower bound to hold for all groups $G=\G$.
As was noted in~\cite{BPS:2012}, several groups of this form do not arise in this way at all.
For example, the group $G_{11,1}$ never occurs as the group of points on any elliptic curve over any finite field.
\end{rmk}

Our final result for $M(\G)$ is that on average the full asymptotic of Theorem~\ref{meg-rephrased} holds unconditionally.

%%%%%%%%%%%%%%%%%%%%%%%%%%%%%%%%%%%%%%%%%%%%%%%%%%%%%%%%%%%%%%%%%%%%%%%%%%%%%%%%%%%%%%%%%%%%%%%%%%%%%%%%%%%%%%%%%%%%%%%%%%%%%%%%%%%%%%%%%%%%%%%%%%%%%%%%%%%%%%%%%%%%%%%%%%%%%%%%%%%%%%%

\begin{thm} \label{CDKS3} Fix $\epsilon>0$ and $A\ge1$. For $2\le x\le y^{1/4-\epsilon}$ we have that
\als{
\frac{1}{xy}\sum_{\substack{m\le x,\, k\le y \\ mk>1}} \left| M(G_{m,k}) - \frac{K(\G) |\G|^2}{|\Aut (\G)|\log|\G|} \right|
	\ll \frac{y}{(\log y)^A},
}
the implied constant depending at most on $A$ and $\epsilon$. Moreover, if the generalized Riemann hypothesis is true, then the same result is true for $x\le y^{1/2-\epsilon}$.
\end{thm}

%%%%%%%%%%%%%%%%%%%%%%%%%%%%%%%%%%%%%%%%%%%%%%%%%%%%%%%%%%%%%%%%%%%%%%%%%%%%%%%%%%%%%%%%%%%%%%%%%%%%%%%%%%%%%%%%%%%%%%%%%%%%%%%%%%%%%%%%%%%%%%%%%%%%%%%%%%%%%%%%%%%%%%%%%%%%%%%%%%%%%%%%%%

In \cite{DS-MEN, DS-MEN-corr}, the second and fourth named authors studied the related question of how many elliptic curves over $\F_p$ have a given number of points, that is to say the asymptotic behaviour of
\[
M(N) :=   \sum_p  \sum_{\substack{ E/\F_p\\ \#E(\F_p)=N}}\frac{1}{|\Aut_p(E)|} .
\]
It was shown in \cite{DS-MEN, DS-MEN-corr} that
\[
M(N)\sim K(N)\cdot \frac{N^2}{\phi(N)\log N}       \quad(N\to\infty)
\]
under suitable assumptions on the distribution of primes in short arithmetic progressions, where
\eq{define K(N)}{
K(N) = \prod_{\ell\nmid N} \left(1- \frac{ \leg{N-1}{\ell}^2\ell+1}{(\ell-1)^2(\ell+1)} \right)
		\prod_{\ell|N} \left(1-\frac{1}{\ell^{\nu_\ell(N)}(\ell-1)}\right).
}
Here $\nu_\ell(N)$ denotes the usual $\ell$-adic valuation of $N$.
As one might expect, the methods of this paper apply to the study of $M(N)$ as well.

We start by recording the obvious identity
\[
M(N) = \sum_{m^2k=N} M(\G) .
\]
Then it is possible to show that, as expected, most of the contribution to $M(N)$ comes from groups $\G$ with $m$ small, that is to say groups that are nearly cyclic.

%%%%%%%%%%%%%%%%%%%%%%%%%%%%%%%%%%%%%%%%%%%%%%%%%%%%%%%%%%%%%%%%%%%%%%%%%%%%%%%%%%%%%%%%%%%%%%%%%%%%%%%%%%%%%%%%%%%%%%%%%%%%%%%%%%%%%%%%%%%%%%%%%%%%%%%%%%%%%%%%%%%%%%%%%%%%%%%%%%%%%%%%%%

\begin{thm}\label{MEN-1} For $N\ge1$ and $x\ge1$, we have that
\[
M(N) = \sum_{\substack{m^2k=N \\ m\le x}} M(\G)
			+ O\left( \frac{N^2}{x\phi(N)\log(2N)} \right) .
\]
\end{thm}

Finally, we conclude with two more results on $M(N)$.

%%%%%%%%%%%%%%%%%%%%%%%%%%%%%%%%%%%%%%%%%%%%%%%%%%%%%%%%%%%%%%%%%%%%%%%%%%%%%%%%%%%%%%%%%%%%%%%%%%%%%%%%%%%%%%%%%%%%%%%%%%%%%%%%%%%%%%%%%%%%%%%%%%%%%%%%%%%%%%%%%%%%%%%%%%%%%%%%%%%%%%%%%%

\begin{thm} \label{MEN-2} Let $N\ge1$ and set
\[
\eta = \frac{1}{N/(\log(2N))}
	\sum_{\substack{ N^-<p\le N^+ \\  p\equiv 1\pmod m }} \sqrt{(p-N^-)(N^+-p)} ,
\]
and note that $\eta\ll1$ by the Brun-Titchmarsch inequality. For any fixed $\lambda>1$,
\[
\eta^\lambda \cdot \frac{N^2}{\phi(N) \log(2N) }  \ll M(N)  \ll \eta^{1/\lambda} \cdot \frac{N^2}{\phi(N)\log(2N)},
\]
the implied constants depending at most on $\lambda$.
\end{thm}

%%%%%%%%%%%%%%%%%%%%%%%%%%%%%%%%%%%%%%%%%%%%%%%%%%%%%%%%%%%%%%%%%%%%%%%%%%%%%%%%%%%%%%%%%%%%%%%%%%%%%%%%%%%%%%%%%%%%%%%%%%%%%%%%%%%%%%%%%%%%%%%%%%%%%%%%%%%%%%%%%%%%%%%%%%%%%%%%%%%%%%%

\begin{thm} \label{MEN-3} Fix $A>0$. For $x\ge1$, we have that
\als{
\frac{1}{x}\sum_{1<N\le x} \left| M(N)   - \frac{K(N)N^2}{\phi(N)\log N} \right|
	\ll_A \frac{x}{(\log x)^A} .
}
\end{thm}

The present paper also includes an appendix (by Greg Martin and the second and fourth named authors) 
giving a probabilistic interpretation to the Euler factors arising in the constants $K(N)$ and $K(G)$ defined by~\eqref{define K(G)} and~\eqref{define K(N)}, respectively.
This interpretation is similar to the heuristic leading to the conjectural constants in related conjectures on properties of the reductions of a fixed global elliptic curve $E$ over the rationals 
(e.g., the Lang-Trotter conjectures~\cite{LT:1976} and the Koblitz~\cite{Kob:1988} conjecture) 
with the additional feature that the Euler factors at the primes $\ell$ dividing $N$ or $|G|$ are related to certain matrix counts over $\Z / \ell^e \Z$ for $e$ large enough.

%%%%%%%%%%%%%%%%%%%%%%%%%%%%%%%%%%%%%%%%%%%%%%%%%%%%%%%%%%%%%%%%%%%%%%%%%%%%%%%%%%%%%%%%%%%%%%%%%%%%%%%%%%%%%%%%%%%%%%%%%%%%%%%%%%%%%%%%%%%%%%%%%%%%%%%%%%%%%%%%%%%%%%%%%%%%%%%%%%%%%%%%%%

\subsection*{Acknowledgements} The work of the second and third authors was partially supported by the Natural Sciences and Engineering Research Council of Canada.
%Discovery Grants ... and 435272-2013, respectively. 
Finally, part of this work was completed while the first, third and fourth authors were postdoctoral fellows at the Centre de recherches math\'ematiques at Montr\'eal, which they would like to thank for the financial support and the pleasant working environment.

%%%%%%%%%%%%%%%%%%%%%%%%%%%%%%%%%%%%%%%%%%%%%%%%%%%%%%%%%%%%%%%%%%%%%%%%%%%%%%%%%%%%%%%%%%%%%%%%%%%%%%%%%%%%%%%%%%%%%%%%%%%%%%%%%%%%%%%%%%%%%%%%%%%%%%%%%%%%%%%%%%%%%%%%%%%%%%%%%%%%%%%%%%

\subsection*{Notation} Given a natural number $n$, we denote with $P^+(n)$ and $P^-(n)$ its largest and smallest prime factor, respectively, with the convention that $P^+(1)=1$ and $P^-(1)=\infty$. Moreover, we let $\tau_r(n)$ denote the coefficient of $1/n^s$ in the Dirichlet series $\zeta(s)^r$. In particular, $\tau_r(n)=r^{\omega(n)}$ for square-free integers $n$, where $\omega(n)$ denotes the number of distinct prime factors of $n$. In the special case when $r=2$, we simply write $\tau(n)$ in place of $\tau_2(n)$, which counts the number of divisors of $n$. We write $f*g$ to denote the Dirichlet convolution of the arithmetic functions $f$ and $g$, defined by $(f*g)(n)=\sum_{ab=n}f(a)g(b)$. As usual, given a Dirichlet character $\chi$, we write $L(s,\chi)$ for its Dirichlet series. In addition, we make use of the notation
\[
E(x,h;q) := \max_{(a,q)=1} \left|  \sum_{\substack{ x <p \le x + h  \\  p\equiv a\pmod q}} \log p  - \frac{h}{\phi(q)} \right| .
\]
Finally, for $d\in\Z$ that is not a square and for $z\ge1$, we let
\[
\L(d) =  L\left(1,\left(\frac{d}{\cdot}\right)\right) = \prod_{\ell} \left( 1- \frac{\leg{d}{\ell}}{\ell}\right)^{-1}
\quad\text{and}\quad
\L(d;z) = \prod_{\ell\le z} \left( 1- \frac{\leg{d}{\ell}}{\ell}\right)^{-1}.
\]

%%%%%%%%%%%%%%%%%%%%%%%%%%%%%%%%%%%%%%%%%%%%%%%%%%%%%%%%%%%%%%%%%%%%%%%%%%%%%%%%%%%%%%%%%%%%%%%%%%%%%%%%%%%%%%%%%%%%%%%%%%%%%%%%%%%%%%%%%%%%%%%%%%%%%%%%%%%%%%%%%%%%%%%%%%%%%%%%%%%%%%%%%%%%%%%%%%%%%%%%%%%%%%%%%%%%%%%%%%%%%%%%%%%%%%%%%%%%%%%%%%%%%%%%%%%%%%%%%%%%%%%%%%%%%%%%%%%%%%%%%%%%%%%%%%%%%%%%%%%%%%%%%%%%%%%%%%%%%%%%%%%%%%%%%%%%%%%%%%%%%%%%%%%%%%%%%%%%%%%%%%%%%%%%%%%%

\section{Outline of the proofs}\label{outline}

In this section, we outline the chief ideas that go into the proofs of our main results.  
However, most of our remarks concern the proofs of Theorems~\ref{CDKS} and~\ref{CDKS3}.
This is primarily because the remaining results are essentially corollaries of these theorems.  
In particular, the main ingredient in the proof of Theorem~\ref{MEN-1} is Theorem~\ref{CDKS}, and the main ingredients in the proof Theorem~\ref{MEN-3} are Theorems~\ref{CDKS3} and~\ref{MEN-1} together with a short computation.  
Theorem~\ref{MEN-2} is not truly a corollary, but its proof is essentially the same as that of Theorem~\ref{CDKS}.
The proof of Theorem~\ref{CDKS2} is somewhat different.  
The ideas involved in its proof are essentially the same as those used to show Theorem 1.6 of~\cite{CDKS1} together with an application of Theorem~\ref{CDKS}. 
All of this will be expounded further in Section~\ref{proofs}, where we complete the proofs of all six results.

%In this section we outline the main ideas that go into the proof of Theorems \ref{CDKS} and \ref{CDKS3}. We will then complete the proofs of these theorems as well as of the other results stated in the introduction in Section \ref{proofs}.
For the remainder of this section, we focus our attention on outlining the main ingredients in the proofs of Theorems~\ref{CDKS} and~\ref{CDKS3}.
Throughout, we fix a group $G=G_{m,k}=\Z/m\Z\times\Z/mk\Z$, and we set $N=|G|=m^2k$. Moreover, given a prime $p\equiv 1\pmod {m}$, we set
\begin{eqnarray} \label{def-dp}
d_{m,k}(p)= \frac{(p-1-N)^2-4N}{m^2}=\left(\frac{p-1}{m}-mk\right)^2-4k.
\end{eqnarray}
Often, when the dependence on $m$ and $k$ is clear from the context, we will simply write $d(p)$ in place of $d_{k,m}(p)$.
Our starting point is the following lemma, whose proof is based on Deuring's work \cite{Deu} and its generalization due to Schoof \cite{Schoof}. We shall give the details of its proof in Section \ref{deuring lemma}.

%%%%%%%%%%%%%%%%%%%%%%%%%%%%%%%%%%%%%%%%%%%%%%%%%%%%%%%%%%%%%%%%%%%%%%%%%%%%%%%%%%%%%%%%%%%%%%%%%%%%%%%%%%%%%%%%%%%%%%%%%%%%%%%%%%%%%%%%%%%%%%%%%%%%%%%%%%%%%%%%%%%%%%%%%%%%%%%%%%%%%%%

\begin{lma}\label{formula for M(G)}
For any $m,k\in\N$, we have that
\[
M(\G) = \sum_{\substack{N^-<p<N^+ \\ p \equiv 1 \pmod m }}\sum_{\substack{f^2\mid d(p) ,\,   (f,k)=1  \\
	d(p)/f^2\equiv 1,0\pmod 4 }}
	\frac{\sqrt{|d(p)|} \L(d(p)/f^2)}{2\pi f} .
\]
\end{lma}

%%%%%%%%%%%%%%%%%%%%%%%%%%%%%%%%%%%%%%%%%%%%%%%%%%%%%%%%%%%%%%%%%%%%%%%%%%%%%%%%%%%%%%%%%%%%%%%%%%%%%%%%%%%%%%%%%%%%%%%%%%%%%%%%%%%%%%%%%%%%%%%%%%%%%%%%%%%%%%%%%%%%%%%%%%%%%%%%%%%%%%%

For the proof of Theorem \ref{CDKS}, we shall use the following simplified but weaker version of Lemma \ref{formula for M(G)}.

%%%%%%%%%%%%%%%%%%%%%%%%%%%%%%%%%%%%%%%%%%%%%%%%%%%%%%%%%%%%%%%%%%%%%%%%%%%%%%%%%%%%%%%%%%%%%%%%%%%%%%%%%%%%%%%%%%%%%%%%%%%%%%%%%%%%%%%%%%%%%%%%%%%%%%%%%%%%%%%%%%%%%%%%%%%%%%%%%%%%%%%

\begin{cor}\label{bounds for M(G)} For any $m,k\in\N$, we have that
\[
\sum_{\substack{N^-<p<N^+ \\ p \equiv 1 \pmod m }}
		\sqrt{|d(p)|} \L(d(p))
				\ll M(\G) \ll
			\sum_{\substack{N^-<p<N^+ \\ p \equiv 1 \pmod m }} \frac{|d(p)|^{3/2}}{\phi(|d(p)|)}
				\L(d(p)).
\]
\end{cor}

%%%%%%%%%%%%%%%%%%%%%%%%%%%%%%%%%%%%%%%%%%%%%%%%%%%%%%%%%%%%%%%%%%%%%%%%%%%%%%%%%%%%%%%%%%%%%%%%%%%%%%%%%%%%%%%%%%%%%%%%%%%%%%%%%%%%%%%%%%%%%%%%%%%%%%%%%%%%%%%%%%%%%%%%%%%%%%%%%%%%%%%

\begin{proof} For the lower bound, note that the term $f=1$ in Lemma \ref{formula for M(G)} always contributes to $M(\G)$, since $d(p)\equiv 0,1\mod 4$ for all $m,k$ and $p\equiv 1\mod m$. For the upper bound, notice that
\[
\L(d(p)/f^2) \le \frac{f}{\phi(f)} \L(d(p)) .
\]
Since
\[
\sum_{f|n} \frac{1}{\phi(f)}\ll \frac{n}{\phi(n)} ,
\]
the claimed upper bound follows.
\end{proof}

Evidently, Lemma \ref{formula for M(G)} and Corollary \ref{bounds for M(G)} reduce the estimation of $M(\G)$ to estimating an average of Dirichlet series evaluated at 1. In order to do so, we expand the Dirichlet series as an infinite sum and invert the order of summation by putting the sum over primes $p$ inside. For each fixed $n$ in the Dirichlet sum,
understanding this sum over primes involves understanding the distribution   of the set 
\eq{our set}{
\left\{\frac{p-1}{m}: N^-<p<N^+,\ p\equiv 1\mod m\right\}
}
in arithmetic progressions $a\mod b$, where the modulus $b=b(n)$ depends on $n$ and other parameters which are less essential.
Already when $b=m=1$, this problem is very hard and unsolved, even if we assume the validity of the Riemann Hypothesis. In order to limit the size of the moduli $b$ that are involved, we need to truncate the Dirichlet series that appear before inverting the order of summation. We could do this for each individual Dirichlet series, using character sum estimates such as the P\'olya-Vinogradov inequality or Burgess's bounds as in \cite{DS-MEN,DS-MEG}, but this would still leave us to deal with rather large moduli $b$. Instead, we use the following result, 
which implies that for {\it most} characters $\chi$, $L(1,\chi)$ can be approximated by a very short Euler product, and then by a sum over integers $n$ supported only on small primes.

%%%%%%%%%%%%%%%%%%%%%%%%%%%%%%%%%%%%%%%%%%%%%%%%%%%%%%%%%%%%%%%%%%%%%%%%%%%%%%%%%%%%%%%%%%%%%%%%%%%%%%%%%%%%%%%%%%%%%%%%%%%%%%%%%%%%%%%%%%%%%%%%%%%%%%%%%%%%%%%%%%%%%%%%%%%%%%%%%%%%%%%

\begin{lma} \label{lemmashortproduct}
Let $\alpha \geq 1$ and $Q\ge3$. There is a set $\mathcal{E}_\alpha(Q)\subset[1,Q]\cap\Z$ of at most $Q^{2/\alpha}$ integers such that if $\chi$ is a Dirichlet character modulo $q\le\exp\{(\log Q)^2\}$ whose conductor does not belong to $\mathcal{E}_\alpha(Q)$, then
\[
L(1,\chi) = \prod_{\ell\le (\log Q)^{8\alpha^2}} \left(1-\frac{\chi(\ell)}{\ell}\right)^{-1} \left(1 + O_\alpha\left(\frac1{(\log Q)^\alpha}\right)\right).
\]
\end{lma}

%%%%%%%%%%%%%%%%%%%%%%%%%%%%%%%%%%%%%%%%%%%%%%%%%%%%%%%%%%%%%%%%%%%%%%%%%%%%%%%%%%%%%%%%%%%%%%%%%%%%%%%%%%%%%%%%%%%%%%%%%%%%%%%%%%%%%%%%%%%%%%%%%%%%%%%%%%%%%%%%%%%%%%%%%%%%%%%%%%%%%%%

\begin{proof} By a classical result, essentially due to Elliott (see \cite[Proposition 2.2]{GS}), we know that there is a set $\mathcal{E}_\alpha(Q)$ of at most $Q^{2/\alpha}$ integers from $[1,Q]$ such that
\[
L(1,\psi) = \prod_{\ell \leq (\log{Q})^{8\alpha^2}} \left( 1 - \frac{\psi(\ell)}{\ell} \right)^{-1}
			\left( 1 + O\left(\frac{\alpha}{(\log Q)^\alpha}\right) \right)
\]
for all primitive characters $\psi$ of conductor in $[1,Q]\setminus\mathcal{E}_\alpha(Q)$. So if $\chi$ is a Dirichlet character modulo $q\le \exp\{(\log Q)^2\}$ induced by $\psi$ and the conductor of $\psi$ is in $[1,Q]\setminus \mathcal{E}_\alpha(Q)$, then
\als{
L(1, \chi)
	&= \prod_{\ell \mid q} \left( 1 - \frac{\psi(\ell)}{\ell} \right)
		\prod_{\ell \leq (\log{Q})^{8\alpha^2}} \left( 1 - \frac{\psi(\ell)}{\ell} \right)^{-1}
		\left( 1 + O\left(\frac{\alpha}{(\log Q)^\alpha}\right) \right)\\
	&= \prod_{\ell \mid q,\,\ell > (\log Q)^{8\alpha^2}  } \left( 1 - \frac{\psi(\ell)}{\ell} \right)
		\prod_{\ell \le (\log Q)^{8\alpha^2}  } \left( 1 - \frac{\chi(\ell)}{\ell} \right)^{-1}
		\left( 1 + O\left(\frac{\alpha}{(\log Q)^\alpha}\right) \right) .
}
Finally, note that
\[
\log\left(\prod_{\ell \mid q,\,\ell > (\log Q)^{8\alpha^2}  }\left( 1 - \frac{\psi(\ell)}{\ell} \right) \right)
	\ll \sum_{\ell \mid q,\,\ell > (\log Q)^{8\alpha^2}  } \frac{1}{\ell}
	\le \frac{\omega(q) }{(\log Q)^{8\alpha^2}} \ll \frac{1}{(\log Q)^{8\alpha^2-2}},
\]
since $\omega(q)\le \log q/\log 2\ll (\log Q)^2$, which completes the proof of the lemma.
\end{proof}

%%%%%%%%%%%%%%%%%%%%%%%%%%%%%%%%%%%%%%%%%%%%%%%%%%%%%%%%%%%%%%%%%%%%%%%%%%%%%%%%%%%%%%%%%%%%%%%%%%%%%%%%%%%%%%%%%%%%%%%%%%%%%%%%%%%%%%%%%%%%%%%%%%%%%%%%%%%%%%%%%%%%%%%%%%%%%%%%%%%%%%%

Expanding the short product in the above lemma leads to an approximation of $L(1,\chi)$ by a sum over $(\log Q)^A$-smooth integers, and we know that very few of them get $>Q^\epsilon$:

%%%%%%%%%%%%%%%%%%%%%%%%%%%%%%%%%%%%%%%%%%%%%%%%%%%%%%%%%%%%%%%%%%%%%%%%%%%%%%%%%%%%%%%%%%%%%%%%%%%%%%%%%%%%%%%%%%%%%%%%%%%%%%%%%%%%%%%%%%%%%%%%%%%%%%%%%%%%%%%%%%%%%%%%%%%%%%%%%%%%%%%%%%%%%%%%%%%%%%%%%%%%%%%%%%%%%

\begin{lma}\label{smooth} Let $f:\mathbb{N}\to\{z\in\C:|z|\le1\}$ be a completely multiplicative function. For $u\ge1$ and $x\ge10$ we have that
\[
\prod_{p\le x}\left(1-\frac{f(p)}p\right)^{-1}
	= \sum_{\substack{P^+(n)\le x\\n\le x^u}} \frac{f(n)}n+ O\left(\frac{\log x}{e^u}\right) .
\]
\end{lma}

%%%%%%%%%%%%%%%%%%%%%%%%%%%%%%%%%%%%%%%%%%%%%%%%%%%%%%%%%%%%%%%%%%%%%%%%%%%%%%%%%%%%%%%%%%%%%%%%%%%%%%%%%%%%%%%%%%%%%%%%%%%%%%%%%%%%%%%%%%%%%%%%%%%%%%%%%%%%%%%%%%%%%%%%%%%%%%%%%%%%%%%%%%%%%%%%%%%%%%%%%%%%%%%%%%%%%

\begin{proof} We have that
\als{
\left| \prod_{p\le x}\left(1-\frac{f(p)}p\right)^{-1}
	- \sum_{\substack{P^+(n)\le x\\n\le x^u}} \frac{f(n)}n \right|
	= \left| \sum_{\substack{P^+(n)\le x\\ n > x^u }} \frac{f(n)}n \right|
	&\le \frac{1}{e^u} \sum_{P^+(n)\le x}  \frac{1}{n^{1-1/\log x}} \\
	&\ll \frac{1}{e^u} \exp\left\{\sum_{p\le x} \frac{1}{p^{1-1/\log x}} \right\} .
}
So using the formula $p^{1/\log x}=1+O(\log p/\log x)$ and the prime number theorem, we obtain the claimed result.
\end{proof}

%%%%%%%%%%%%%%%%%%%%%%%%%%%%%%%%%%%%%%%%%%%%%%%%%%%%%%%%%%%%%%%%%%%%%%%%%%%%%%%%%%%%%%%%%%%%%%%%%%%%%%%%%%%%%%%%%%%%%%%%%%%%%%%%%%%%%%%%%%%%%%%%%%%%%%%%%%%%%%%%%%%%%%%%%%%%%%%%%%%%%%%%%%%%%%%%%%%%%%%%%%%%%%%%%%%%%

Combining Lemmas \ref{lemmashortproduct} and \ref{smooth}, we may replace $L(1,\chi)$ by a very short sum for most characters $\chi$, which means that we only need information for the distribution of the set \eqref{our set} for very small moduli. This leads to the following fundamental result, which is an improvement of Theorem \ref{meg-rephrased}. It will be proven in Section \ref{approx}.

%%%%%%%%%%%%%%%%%%%%%%%%%%%%%%%%%%%%%%%%%%%%%%%%%%%%%%%%%%%%%%%%%%%%%%%%%%%%%%%%%%%%%%%%%%%%%%%%%%%%%%%%%%%%%%%%%%%%%%%%%%%%%%%%%%%%%%%%%%%%%%%%%%%%%%%%%%%%%%%%%%%%%%%%%%%%%%%%%%%%%%%

\begin{thm}\label{approximation of M(G)} Fix $\alpha\ge1$ and $\epsilon\le 1/3$, and consider integers $m$ and $k$ with $1\le m\le k^{\alpha}$ and $k$ large enough so that
$k^{\frac{1}{2}-\epsilon} \ge(\log k)^{\alpha+2}$. Set $G=G_{m,k}$, and consider $h\in[mk^{\epsilon},m \sqrt{k}/(\log k)^{\alpha+2}]$. Then
\als{
M(G) &=  \frac{K(G) |G|^2}{|\Aut (G)|\log|G|}
	 + O_{\alpha,\epsilon} \left(\frac{k}{(\log k)^{\alpha}}
	 	+ \frac{\sqrt{k}}{h} \sum_{q\le k^{\epsilon} } \tau_3(q)
				 \int_{N^-}^{N^+} E(y,h; qm ) \dee y\right),
}
where $K(G)$ is defined by \eqref{define K(G)}.
\end{thm}

%%%%%%%%%%%%%%%%%%%%%%%%%%%%%%%%%%%%%%%%%%%%%%%%%%%%%%%%%%%%%%%%%%%%%%%%%%%%%%%%%%%%%%%%%%%%%%%%%%%%%%%%%%%%%%%%%%%%%%%%%%%%%%%%%%%%%%%%%%%%%%%%%%%%%%%%%%%%%%%%%%%%%%%%%%%%%%%%%%%%%%%

Even though we cannot estimate the error term for any given values of $m$ and $k$, we can do so if we average over $m$ and $k$ using the following result, which is a consequence of Theorem 1.1 in \cite{Kou}.

%%%%%%%%%%%%%%%%%%%%%%%%%%%%%%%%%%%%%%%%%%%%%%%%%%%%%%%%%%%%%%%%%%%%%%%%%%%%%%%%%%%%%%%%%%%%%%%%%%%%%%%%%%%%%%%%%%%%%%%%%%%%%%%%%%%%%%%%%%%%%%%%%%%%%%%%%%%%%%%%%%%%%%%%%%%%%%%%%%%%%%%

\begin{lma}\label{bv-short}Fix $\epsilon>0$ and $A\ge1$. For $x\ge h\ge2$ and $1\le Q^2\le h/x^{1/6+\epsilon}$, we have that
\[
\int_x^{2x}\sum_{q\le Q} E(y,h;q) \dee y \ll\frac{xh}{(\log x)^A}.
\]
If, in addition, the Riemann hypothesis for Dirichlet $L$-functions is true, then the above estimate holds when $1\le Q^2\le h/x^\epsilon$.
\end{lma}

%%%%%%%%%%%%%%%%%%%%%%%%%%%%%%%%%%%%%%%%%%%%%%%%%%%%%%%%%%%%%%%%%%%%%%%%%%%%%%%%%%%%%%%%%%%%%%%%%%%%%%%%%%%%%%%%%%%%%%%%%%%%%%%%%%%%%%%%%%%%%%%%%%%%%%%%%%%%%%%%%%%%%%%%%%%%%%%%%%%%%%%

Theorem \ref{approximation of M(G)} and Lemma \ref{bv-short} lead to a proof of Theorem \ref{CDKS3} in a fairly straightforward way as we will see in Section \ref{proofs}.

Next, we turn to the proof of Theorem \ref{CDKS}. 
Using Corollary~\ref{bounds for M(G)} and H\"older's inequality, we reduce the proof of this result to that of controlling sums of the form
\eq{CDKS - key sum}{
\sum_{\substack{N^-<p<N^+ \\ p \equiv 1 \pmod m }}
	\left(\frac{|d(p)|}{\phi(|d(p)|) }\right)^s \L(d(p))^r ,
}
where we take $r>0$ to prove the implicit upper bound and $r<0$ for the lower bound.
%\fcom{(Is there a typo here? Should it be "$s \geq 0$" and $r \in \mathbb R$?)}
Nevertheless, we only seek an upper bound for the sum in \eqref{CDKS - key sum}, even for the lower bound in Theorem \ref{CDKS}. Therefore we can replace the sum over primes with a sum over almost primes and use sieve methods to detect the latter kind of integers. More precisely, we will majorize the characteristic function of primes $\le 2N$ by a convolution $\lambda*1$, where $\lambda$ is a certain truncation of the M\"obius function. This will be done using the {\it fundamental lemma of sieve methods}, which we state below in the form found in \cite[Lemma 5]{FI}. We could have also used Selberg's sieve, but the calculations are actually simpler when using Lemma \ref{lemma-FI}.

%%%%%%%%%%%%%%%%%%%%%%%%%%%%%%%%%%%%%%%%%%%%%%%%%%%%%%%%%%%%%%%%%%%%%%%%%%%%%%%%%%%%%%%%%%%%%%%%%%%%%%%%%%%%%%%%%%%%%%%%%%%%%%%%%%%%%%%%%%%%%%%%%%%%%%%%%%%%%%%%%%%%%%%%%%%%%%%%%%%%%%%%%%

\begin{lma}\label{lemma-FI}
Let $y\ge2$ and $D=y^u$ with $u\ge2$. There exist two arithmetic functions $\lambda^\pm:\mathbb{N}\to[-1,1]$, supported on $\{d\in\mathbb{N}:P^+(d)\le y,\,d\le D\}$, for which
\[
\begin{cases}
	(\lambda^-*1)(n)=(\lambda^+*1)(n)=1		&\text{if}\ P^-(n)>y,\\
	(\lambda^-*1)(n)\le0\le(\lambda^+*1)(n)		&\text{otherwise}.
\end{cases}
\]
Moreover, if $g:\mathbb{N}\to\mathbb{R}$ is a multiplicative function with $0\le g(p)\le\min\{2,p-1\}$ for all primes $p\le y$, and $\lambda\in\{\lambda^+,\lambda^-\}$, then
\[
\sum_d  \frac{ \lambda(d)g(d) }{d}
	=  (1+O(e^{-u})) \prod_{p\le y} \left(1-\frac{g(p)}p\right).
\]
\end{lma}

%%%%%%%%%%%%%%%%%%%%%%%%%%%%%%%%%%%%%%%%%%%%%%%%%%%%%%%%%%%%%
%%%%%%%%%%%%%%%%%%%%%%%%%%%%%%%%%%%%%%%%%%%%%%%%%%%%%%%%%%%%%
%%%%%%%%%%%%%%%%%%%%%%%%%%%%%%%%%%%%%%%%%%%%%%%%%%%%%%%%%%%%%

Combining Lemmas \ref{lemmashortproduct} and \ref{lemma-FI}, we are led to following key result, which will proven in Section \ref{proof of Prop startwiththat}. As we will see in the same section, Theorem \ref{CDKS} is an easy consequence of this intermediate result.

%%%%%%%%%%%%%%%%%%%%%%%%%%%%%%%%%%%%%%%%%%%%%%%%%%%%%%%%%%%%%%%%%%%%%%%%%%%%%%%%%%%%%%%%%%%%%%%%%%%%%%%%%%%%%%%%%%%%%%%%%%%%%%%%%%%%%%%%%%%%%%%%%%%%%%%%%%%%%%%%%%%%%%%%%%%%%%%%%%%%%%%%%%

\begin{prop}\label{startwiththat} Let $m,k\in\N$ and set $N=m^2k$. For any $r\in\mathbb{R}$ and $s\ge0$, we have that
\[
\sum_{\substack{N^-<p<N^+ \\ p \equiv 1 \pmod m }}
	\left(\frac{|d(p)|}{\phi(|d(p)|) }\right)^s \L(d(p))^r
	  \ll_{r,s} \left(\frac{k}{\phi (k)}\right)^r \frac{\sqrt{N}}{\phi(m)\log(2k)}.
\]
\end{prop}

%%%%%%%%%%%%%%%%%%%%%%%%%%%%%%%%%%%%%%%%%%%%%%%%%%%%%%%%%%%%%%%%%%%%%%%%%%%%%%%%%%%%%%%%%%%%%%%%%%%%%%%%%%%%%%%%%%%%%%%%%%%%%%%%%%%%%%%%%%%%%%%%%%%%%%%%%%%%%%%%%%%%%%%%%%%%%%%%%%%%%%%%%%%%%%%%%%%%%%%%%%%%%%%%%%%%%%%%%%%%%%%%%%%%%%%%%%%%%%%%%%%%%%%%%%%%%%%%%%%%%%%%%%%%%%%%%%%%%%%%%%%%%%%%%%%%%%%%%%%%%%%%%%%%%%%%%%%%%%%%%%%%%%%%%%%%%%%%%%%%%%%%%%%%%%%%%%%%%%%%%%%%%%%%%%%%

\section{Completion of the proof of the main results}\label{proofs}

In this section we prove Theorems \ref{CDKS}-\ref{MEN-3}. We start by stating a preliminary result, which is Lemma 15 of \cite{DS-MEG} in slightly altered form.

%%%%%%%%%%%%%%%%%%%%%%%%%%%%%%%%%%%%%%%%%%%%%%%%%%%%%%%%%%%%%%%%%%%%%%%%%%%%%%%%%%%%%%%%%%%%%%%%%%%%%%%%%%%%%%%%%%%%%%%%%%%%%%%%%%%%%%%%%%%%%%%%%%%%%%%%%%%%%%%%%%%%%%%%%%%%%%%%%%%%%%%

\begin{lma}\label{Aut(G)} For $m,k\in\N$, we have that
\[
\frac{|{\Aut}(G_{m,k})|}{|G_{m,k}|}
	= m\phi(m)  \frac{\phi(k)}{k}  \prod_{\substack{\ell | m \\ \ell\nmid k}} \left(1-\frac{1}{\ell^2} \right)  .
\]
\end{lma}

%%%%%%%%%%%%%%%%%%%%%%%%%%%%%%%%%%%%%%%%%%%%%%%%%%%%%%%%%%%%%%%%%%%%%%%%%%%%%%%%%%%%%%%%%%%%%%%%%%%%%%%%%%%%%%%%%%%%%%%%%%%%%%%%%%%%%%%%%%%%%%%%%%%%%%%%%%%%%%%%%%%%%%%%%%%%%%%%%%%%%%%

\begin{proof}[Proof of Theorem \ref{CDKS}] The claimed inequalities are a consequence of Corollary \ref{bounds for M(G)}, Proposition \ref{startwiththat}, and H\"older's inequality.
Indeed, let $\mu=\lambda/(\lambda-1)$, so that $1/\lambda+1/\mu=1$. Then we have that
\als{
M(\G)
	&\ll \sum_{\substack{ N^-<p<N^+ \\ p\equiv 1\mod m}} \sqrt{|d(p)|} \frac{|d(p)|}{\phi(|d(p)|)} \L(d(p)) \\
	&\le \left( \sum_{\substack{ N^-<p<N^+ \\ p\equiv 1\mod m}} \sqrt{|d(p)|}\right)^{\frac{1}{\lambda}}
		\left( \sum_{\substack{ N^-<p<N^+ \\ p\equiv 1\mod m}}
			\sqrt{|d(p)|} \left(\frac{|d(p)|}{\phi(|d(p)|)}\right)^\mu \L(d(p))^{\mu}\right)^{\frac{1}{\mu}}  \\
	&\ll
	 \left( \sum_{\substack{ N^-<p<N^+ \\ p\equiv 1\mod m}} \frac{\sqrt{(N^+-p)(p-N^-)}}{m} \right)^{\frac{1}{\lambda}}
		\left( \sum_{\substack{ N^-<p<N^+ \\ p\equiv 1\mod m}}
			\sqrt{k} \left(\frac{|d(p)|}{\phi(|d(p)|)}\right)^\mu \L(d(p))^{\mu}\right)^{\frac{1}{\mu}}  ,
}
since $|d(p)|=(N^+-p)(p-N^-)/m^2\ll N/m^2=k$. So the definition of $\delta$ and Proposition \ref{startwiththat} imply that
\[
M(\G) \ll_{\lambda,A} \delta^{1/\lambda} \frac{km}{\phi(m)\log(2N)} \frac{k}{\phi(k)}.
\]
Hence the upper bound in Theorem \ref{CDKS} follows by Lemma \ref{Aut(G)}.

The proof of the lower bound is similar, having as a starting point the inequality
\[
\sum_{\substack{ N^-<p<N^+ \\ p\equiv 1\mod m}} \sqrt{|d(p)|}
	\le \left( \sum_{\substack{ N^- <p \le N^+  \\ p\equiv 1\pmod{m} }}
			\sqrt{|d(p)|} \L(d(p)) \right)^{\frac{1}{\lambda}}
		\left( \sum_{\substack{ N^- <p \le N^+ \\ p\equiv 1\pmod{m} }}
			\frac{\sqrt{|d(p)|}}{\L(d(p))^{\mu/\lambda}} \right)^{\frac{1}{\mu}} .
\]
\end{proof}

%%%%%%%%%%%%%%%%%%%%%%%%%%%%%%%%%%%%%%%%%%%%%%%%%%%%%%%%%%%%%%%%%%%%%%%%%%%%%%%%%%%%%%%%%%%%%%%%%%%%%%%%%%%%%%%%%%%%%%%%%%%%%%%%%%%%%%%%%%%%%%%%%%%%%%%%%%%%%%%%%%%%%%%%%%%%%%%%%%%%%%%

\begin{proof}[Proof of Theorem \ref{MEN-2}]
The proof of Theorem \ref{MEN-2} is completely analogous to the proof of Theorem \ref{CDKS}. The only difference is that instead of starting with Corollary \ref{bounds for M(G)}, we observe that
\[
\sum_{N^-<p<N^+} \sqrt{|D_N(p)|} \L(D_N(p))
				\ll M(N) \ll
			\sum_{N^-<p<N^+} \frac{|D_N(p)|^{3/2}}{\phi(|D_N(p)|)}
				\L(D_N(p)),
\]
a consequence of relation~\eqref{reduction to class number avg} below with $n=1$.
\end{proof}

%%%%%%%%%%%%%%%%%%%%%%%%%%%%%%%%%%%%%%%%%%%%%%%%%%%%%%%%%%%%%%%%%%%%%%%%%%%%%%%%%%%%%%%%%%%%%%%%%%%%%%%%%%%%%%%%%%%%%%%%%%%%%%%%%%%%%%%%%%%%%%%%%%%%%%%%%%%%%%%%%%%%%%%%%%%%%%%%%%%%%%%

\begin{proof}[Proof of Theorem~\ref{CDKS2}] Note that when $m=k=1$ and $N=1$, then $N^+=4$ and $N^-=0$ and thus the primes 2 and 3 belong to the set $\{N^-<p\le N^+: p\equiv 1\pmod m\}$. So, by Theorem \ref{CDKS}, it suffices to show Theorem \ref{CDKS2} when $y$ is large enough. We further assume that $x\in\N$, which we may certainly do. Observe that $(N^+-p)(p-N^-)\asymp N$ for $p\in((\sqrt{N}-1/2)^2,(\sqrt{N}+1/2)^2)$, and thus
\[
\frac{1}{N/(\phi(m)\log(2N))} \sum_{\substack{N^-<p<N^+ \\ p\equiv 1\mod{p} }} \sqrt{ (N^+-p)(p-N^-)}
	\gg \frac{\phi(m)}{\sqrt{N}} \sum_{\substack{(\sqrt{N}-1/2)^2<p<(\sqrt{N}+1/2)^2 \\ p\equiv 1\mod{m} }} \log p .
\]
So, if we set
\[
C(m,k) = \frac{|\G|^2}{|\Aut(\G)| \log(2\G)} \asymp \frac{mk^2}{\phi(m)\phi(k)\log(mk)},
\]	
then Theorem \ref{CDKS} with $\lambda=2$ implies that
\als{
\sum_{\substack{3x/4<m\le x \\ y/100<k\le y}} \sqrt{ \frac{M(\G)}{C(m,k)} }
	&\gg \sum_{\substack{3x/4<m\le x \\ y/100<k\le y}} \frac{\phi(m)}{x\sqrt{y}}
		\sum_{\substack{ (m\sqrt{k}-1/2)^2<p< (m\sqrt{k}+1/2)^2 \\ p\equiv 1\pmod m}} \log p \\
	&\ge\sum_{3x/4<m\le x} \sum_{\substack{ x^2y/3 < p \le 4x^2y/9 \\ p\equiv 1\pmod m }} \frac{\phi(m)\log p}{x\sqrt{y}}\sum_{ \substack{y/100<k \le y \\ (\sqrt p -1/2)^2/m^2<k < (\sqrt p+1/2)^2/m^2 }} 1 ,
}
provided that $y$ is large enough. Note that
\[
\frac{(\sqrt p+1/2)^2 -  (\sqrt p-1/2)^2}{m^2} = \frac{2\sqrt{p}}{m^2} \ge \frac{2x\sqrt{y/3}}{x^2} > 1,
\]
by our assumptions that $x\le\sqrt{y}$. Since we also have that $(\sqrt p-1/2)^2/m^2>y/100$ and that $(\sqrt p+1/2)^2/m^2\le y$ for $y$ large enough and $m$ and $p$ as above, we conclude that
\[
\sum_{\substack{3x/4<m\le x \\ y/100<k\le y}} \sqrt{ \frac{M(\G)}{C(m,k)} }
	\gg \frac{1}{x^2} \sum_{3x/4<m\le x} \phi(m)
		\sum_{\substack{ x^2y/3 < p \le 4x^2y/9 \\ p\equiv 1\pmod m }} \log p.
\]
This last double sum equals
\[
\sum_{3x/4<m\le x} \phi(m) \cdot
		\frac{x^2y}{9\phi(m)} + O_A\left( \frac{x^3y}{(\log y)^A}\right) \gg x^3y,
\]
by the Bombieri Vinogradov theorem. Therefore we conclude that
\[
\sum_{\substack{3x/4<m\le x \\ y/100<k\le y}} \sqrt{ \frac{M(G_{m,k})}{C(m,k)} } \gg xy.
\]
Since the summands are all $\ll1$ in this range by Theorem \ref{CDKS} (recall that $\delta\ll1$ there), we obtain Theorem \ref{CDKS2}.
\end{proof}

%%%%%%%%%%%%%%%%%%%%%%%%%%%%%%%%%%%%%%%%%%%%%%%%%%%%%%%%%%%%%%%%%%%%%%%%%%%%%%%%%%%%%%%%%%%%%%%%%%%%%%%%%%%%%%%%%%%%%%%%%%%%%%%%%%%%%%%%%%%%%%%%%%%%%%%%%%%%%%%%%%%%%%%%%%%%%%%%%%%%%%%

\begin{proof}[Proof of Theorem \ref{CDKS3}] Let $\theta$ be a parameter, which we take to be $1/2$ or $1/4$, according to whether we assume the generalized Riemann hypothesis or not. We then suppose that $1\le x\le y^{\theta-\epsilon}$. Note that Theorem \ref{CDKS} and Lemma \ref{Aut(G)} imply that
\[
\sum_{\substack{m\le x,\, k\le y/(\log y)^A \\ mk>1}} \left| M(G_{m,k}) - \frac{K(\G) |\G|^2}{|\Aut (\G)|\log|\G|} \right|
	\ll \frac{xy^2}{(\log y)^A} .
\]
We break the remaining range of $m$ and $k$ into dyadic intervals, hence reducing Theorem \ref{CDKS3} to showing that
\[
E:=  \sum_{\substack{ x/2< m\le x \\ y/2< k\le y}}
		\left| M(G_{m,k}) - \frac{K(\G) |\G|^2}{|\Aut (\G)|\log|\G|} \right| \ll_{\epsilon,A} \frac{xy^2}{(\log y)^A}
\]
for $x\le y^{\theta-\epsilon}$. (Note that these might be different values of $x,y$ and $\epsilon$ than the ones we started with.) We apply Theorem \ref{approximation of M(G)} with $h= (x^2y)^{1/2}/(\log y)^{A+2}$ for all $m\in[x/2,x]$ and $k\in[y/2,y]$, to deduce that
\als{
E &\ll \frac{\sqrt{y}}{h} \sum_{\substack{ x/2< m\le x \\ y/2< k\le y}} \sum_{q \le k^{\epsilon} } \tau_3(q)
				 \int_{(m^2k)^-}^{(m^2k)^+} E(t,h; qm ) \dee t
				+  \frac{xy^2}{(\log y)^A} \\
	&=: E'+  \frac{xy^2}{(\log y)^A} ,
}
say. Putting the sum over $k$ inside, we find that
\als{
E'&\ll \frac{\sqrt{y}}{h} \sum_{x/2<m\le x}  \sum_{q \le y^{\epsilon} } \tau_3(q)
	\int_{x^2y/10}^{2x^2y}  E(t,h; qm) 
		\left(\sum_{ \substack{ y/2<k\le y \\ t^-/m^2<k < t^+/m^2}} 1  \right) \dee t \\
	&\ll \frac{y}{hx} \sum_{m\le x}  \sum_{q \le y^{\epsilon} } \tau_3(q)
	\int_{x^2y/10}^{2x^2y}  E(t,h; qm)  \dee t 
	\le \frac{y}{hx} \sum_{m\le x}  \sum_{q \le y^{\epsilon} } \tau_4(q)
	\int_{x^2y/10}^{2x^2y}  E(t,h; q)  \dee t .
}
We note that $E(u,h; b) \ll \sqrt{h/\phi(b)} \sqrt{E(u,h;b)}$, by the Brun-Titchmarsch inequality. So the Cauchy-Schwarz inequality and Lemma \ref{bv-short} imply that
\als{
E'&\ll \frac{y}{x h} \left(\sum_{b\le xy^{3\epsilon}} \tau_4(b)^2 \int_{x^2y/10}^{2x^2y}
		\frac{h}{\phi(b)}  \dee t\right)^{\frac{1}{2}}
		\left( \sum_{b\le xy^{3\epsilon}} \int_{x^2y/10}^{2x^2y} E(t,h; b) \dee t\right)^{\frac{1}{2}} \\
	&\ll \frac{y}{x h} \left( x^2yh(\log y)^{16} \cdot \frac{x^2yh}{(\log y)^{2A+16}}  \right)^{\frac{1}{2}}
		= \frac{xy^2}{(\log y)^A} ,
}
which completes the proof of Theorem \ref{CDKS3}.
\end{proof}

%%%%%%%%%%%%%%%%%%%%%%%%%%%%%%%%%%%%%%%%%%%%%%%%%%%%%%%%%%%%%%%%%%%%%%%%%%%%%%%%%%%%%%%%%%%%%%%%%%%%%%%%%%%%%%%%%%%%%%%%%%%%%%%%%%%%%%%%%%%%%%%%%%%%%%%%%%%%%%%%%%%%%%%%%%%%%%%%%%%%%%%

\begin{proof}[Proof of Theorem \ref{MEN-1}]
%Proposition \ref{startwiththat} and Corollary \ref{bounds for M(G)} imply that
Theorem~\ref{CDKS} implies that
\[
M(\G) \ll \frac{k^{3/2}}{\phi (k)} \frac{\sqrt{N}}{\phi(m)\log(2k)} = \frac{mk^2}{\phi(k)\phi(m)\log(2k)}\le \frac{Nmk}{\phi(N)\phi(m)\log(2k)}.
\]
Therefore,
\[
\sum_{\substack{m^2k=N \\ m>x}} M(\G)
	\ll \sum_{\substack{m^2|N \\ x<m\le\sqrt{N} }} \frac{N^2}{m\phi(m)\phi(N)\log(2N/m^2)}
	\ll \frac{N^2}{x\phi(N)\log(2N)},
\]
which completes the proof of Theorem \ref{MEN-1}.
\end{proof}

%%%%%%%%%%%%%%%%%%%%%%%%%%%%%%%%%%%%%%%%%%%%%%%%%%%%%%%%%%%%%%%%%%%%%%%%%%%%%%%%%%%%%%%%%%%%%%%%%%%%%%%%%%%%%%%%%%%%%%%%%%%%%%%%%%%%%%%%%%%%%%%%%%%%%%%%%%%%%%%%%%%%%%%%%%%%%%%%%%%%%%%

\begin{proof}[Proof of Theorem \ref{MEN-3}]
In view of Theorem \ref{MEN-1}, it suffices to show that
\[
\sum_{1<N\le x} \left| \sum_{\substack{ m^2k=N \\ m\le (\log x)^A}} M(\G)  - \frac{K(N)N^2}{\phi(N)\log N} \right|
	\ll_A \frac{x^2}{(\log x)^A},
\]
where $K(N)$ is defined by~\eqref{define K(N)}.
Note that
\als{
&\sum_{1<N\le x} \left| \sum_{\substack{ m^2k=N \\ m\le (\log x)^A}} M(\G)
			- \sum_{\substack{ m^2k=N \\ m\le (\log x)^A}}  \frac{K(\G) |\G|^2}{|\Aut (\G)|\log|\G|} \right|   \\
	&\quad\le \sum_{\substack{1<m^2k\le x \\ m\le (\log x)^A}}
			 \left| M(\G)  - \frac{K(\G) |\G|^2}{|\Aut (\G)|\log|\G|} \right|   \\
	&\quad\le \sum_{1\le 2^j\le (\log x)^{A} }
		\sum_{\substack{k\le x/4^j \\ 2^j\le m<2^{j+1} \\ m^2k>1 }}
			 \left| M(\G)  - \frac{K(\G) |\G|^2}{|\Aut (\G)|\log|\G|} \right|   \\
	&\quad\ll_A \sum_{1\le 2^j\le (\log x)^{A}} \frac{x^2}{8^j(\log x)^A} \ll \frac{x^2}{(\log x)^A}
}
by Theorem~\ref{CDKS3}.
So it suffices to show that
\eq{MEN-3 goal}{
\sum_{1<N\le x} \frac{N}{\log N} \left| \sum_{\substack{ m^2k=N \\ m\le (\log x)^A}} \frac{K(\G) |\G|}{|\Aut (\G)|}
	- \frac{K(N)N}{\phi(N)} \right|
	\ll_A \frac{x^2}{(\log x)^A} .
}
In fact, Lemma \ref{Aut(G)} implies that
\als{
\frac{ K(\G) |\G|}{|\Aut(\G)|}
	&= \frac{k}{m\phi(m)\phi(k)}  \prod_{\substack{\ell | m \\ \ell\nmid k}} \left(1-\frac{1}{\ell^2} \right)^{-1}  K(\G) \\
	&= \frac{N}{m^2\phi(N)}\prod_{\ell|(m,k)} \left(1-\frac{1}{\ell}\right)^{-1}
		\prod_{\substack{\ell | m \\ \ell\nmid k}} \left(1-\frac{1}{\ell^2} \right)^{-1}  K(\G)   \\
	&= \frac{N}{m^2\phi(N)}  \prod_{\ell\nmid N}\left(1-\frac{\leg{N-1}{\ell}^2\ell+1}{(\ell-1)^2(\ell+1)}\right)
		\prod_{\ell\mid (m,k)}\left(1+\frac{1}{\ell}\right)
		\prod_{\substack{\ell\mid k\\ \ell\nmid m}}\left(1-\frac{1}{\ell(\ell-1)}\right) .
}
Therefore,
\als{
\sum_{\substack{m^2k = N \\ m\le (\log x)^A}} \frac{K(\G) |\G|}{|\Aut(\G)|}
&= \sum_{m^2k = N} K(\G) \frac{|\G|}{|\Aut(\G)|}  + O\left(\frac{N}{(\log x)^A\phi(N)}\right) \\
&= \frac{N}{\phi(N)}  \prod_{\ell\nmid N}\left(1-\frac{\leg{N-1}{\ell}^2\ell+1}{(\ell-1)^2(\ell+1)}\right) \cdot S(N)
	+ O\left(\frac{N}{(\log x)^A\phi(N)}\right),
}
where
\[
S(N) = \sum_{m^2k=N} \frac{1}{m^2} \prod_{\ell\mid (m,k)}\left(1+\frac{1}{\ell}\right)
		\prod_{\substack{\ell\mid k\\ \ell\nmid m}}\left(1-\frac{1}{\ell(\ell-1)}\right) .
\]
Note that
\als{
S(\ell^v) &= 1-\frac{1}{\ell(\ell-1)} + \sum_{1\le j\le v/2} \frac{1}{\ell^{2j}} \left(1+\frac{{\bf 1}_{j<v/2}}{\ell}\right) \\
	&=1-\frac{1}{\ell(\ell-1)} +  \sum_{1\le j\le v/2} \frac{1}{\ell^{2j}}
		+  \sum_{1\le j\le v/2} \frac{{\bf 1}_{j<v/2}}{\ell^{2j+1}} \\
	&=1-\frac{1}{\ell(\ell-1)} + \sum_{i=2}^v \frac{1}{\ell^i} =  1-\frac{1}{\ell^v(\ell-1)} .
}
So we conclude that
\[
\sum_{\substack{m^2k = N \\ m\le (\log x)^A}} \frac{K(\G) |\G|}{|\Aut(\G)|}
	= \frac{K(N) N}{\phi(N)} + O\left(\frac{N}{(\log x)^A\phi(N)}\right) ,
\]
which yields relation \eqref{MEN-3 goal}, thus completing the proof of Theorem \ref{MEN-3}.
\end{proof}

%%%%%%%%%%%%%%%%%%%%%%%%%%%%%%%%%%%%%%%%%%%%%%%%%%%%%%%%%%%%%%%%%%%%%%%%%%%%%%%%%%%%%%%%%%%%%%%%%%%%%%%%%%%%%%%%%%%%%%%%%%%%%%%%%%%%%%%%%%%%%%%%%%%%%%%%%%%%%%%%%%%%%%%%%%%%%%%%%%%%%%%%%%%%%%%%%%%%%%%%%%%%%%%%%%%%%%%%%%%%%%%%%%%%%%%%%%%%%%%%%%%%%%%%%%%%%%%%%%%%%%%%%%%%%%%%%%%%%%%%%%%%%%%%%%%%%%%%%%%%%%%%%%%%%%%%%%%%%%%%%%%%%%%%%%%%%%%%%%%%%%%%%%%%%%%%%%%%%%%%%%%%%%%%%%%%

\section{Reduction to an average of Dirichlet series}\label{deuring lemma}

In this section, we prove Lemma \ref{formula for M(G)} using the theory developed by Deuring~\cite{Deu} and somewhat generalized by Schoof~\cite{Schoof}. As before, we fix a group $G=G_{m,k}=\Z/m\Z\times\Z/mk\Z$, and we set $N=|G|=m^2k$. Given a prime $p$ and an integer $n$ such that $n^2|N$, we define
\[
M_p(N;n)=\sum_{\substack{E/\F_p\\ \#E(\F_p)=N\\ E(\F_p)[n]\isom G_{n,1}}}\frac{1}{|\Aut_p(E)|},
\]
the weighted number of isomorphism classes of elliptic curves over any prime finite field which have exactly $N$ rational points and whose rational $n$-torsion subgroup is isomorphic to $G_{n,1}=\Z/n\Z\times\Z/n\Z$.
It is not hard to relate $M_p(G)$ to a sum involving $M_p(N;n)$.
This is accomplished via an inclusion-exclusion argument, which gives the relation
\begin{equation}\label{inclusion exclusion}
	M_p(G)=\sum_{r^2 | k} \mu(r)M_p(N; r m).
\end{equation}

In~\cite{Schoof}, Schoof essentially gave a formula for $M_p(N;n)$ in terms of class numbers.
However, one needs to exercise care here as Schoof counts each $\F_p$-isomorphism class $E$ with weight $1$ instead of with weight $1/|\Aut_p(E)|$ as we do here.
Given a negative discriminant $D$, we let $H(D)$ denote the \textit{Kronecker class number}, which is defined as
\[
H(D)=\sum_{ \substack{f^2\mid D\\ D/f^2 \equiv 0,1\pmod{4} }}  \frac{h(D/f^2)}{w(D/f^2)}.
\]
Here, as usual, $h(d)$ denotes the (ordinary) class number of the unique imaginary quadratic order of discriminant $d$,
and $w(d)$ denotes the cardinality of its unit group.
Then letting
\[
D_N(p)=(p+1-N)^2-4p=(p-1-N)^2-4N
\]
and reworking the proofs of~\cite[Lemma 4.8 and Theorem 4.9]{Schoof} to count each class $E$ with weight $1/|\Aut_p(E)|$,
we arrive at the formula
\eq{reduction to class number avg}{
M_p(N;n)=
	\begin{cases}
		H\left(\frac{D_N(p)}{n^2}\right)&\text{if }p\in(N^-,N^+)\text{ and }p\equiv 1\pmod n,\\
		0&\text{otherwise}.
		\end{cases}
}
Note here that $D_N(p)/n^2$ is a negative discriminant whenever $p\in(N^-,N^+)$, $p\equiv 1\pmod n$, and $n^2\mid N$.

%%%%%%%%%%%%%%%%%%%%%%%%%%%%%%%%%%%%%%%%%%%%%%%%%%%%%%%%%%%%%%%%%%%%%%%%%%%%%%%%%%%%%%%%%%%%%%%%%%%%%%%%%%%%%%%%%%%%%%%%%%%%%%%%%%%%%%%%%%%%%%%%%%%%%%%%%%%%%%%%%%%%%%%%%%%%%%%%%%%%%%%

\begin{lma}\label{Deuring for groups}
Let $m,k\in\N$ %\ccom{new sentence follows}
and recall that $d(p) = d_{m,k}(p)$ is defined by \eqref{def-dp}. If $p\in (N^-,N^+)$ and $p\equiv 1\pmod m$, then
\[
M_p(\G)=
\sum_{\substack{f^2\mid d(p),\, (f,k)=1\\ \frac{d(p)}{f^2}\equiv 0,1\pmod 4}}\frac{h(d(p)/f^2)}{w(d(p)/f^2)}.
\]
Otherwise, $M_p(\G)=0$.
\end{lma}

\begin{rmk}
The above formula is amenable to computation.
Indeed, given a prime $p$ and any $m$ and $k$,
very simple modifications to the usual quadratic forms algorithm for computing class numbers (see~\cite[pp.~99--100]{BV:2007} for example)
make it possible to compute $M_p(\G)$ using at most $O(k)$ arithmetic operations, which is reasonable for small $k$.
If we put
\begin{equation*}
H_k(D)=\sum_{\substack{f^2\mid D,\, (f,k)=1\\ \frac{D}{f^2}\equiv 0,1\pmod 4}}\frac{h(D/f^2)}{w(D/f^2)}
\end{equation*}
for each negative discriminant $D$ and each positive integer $k$, then the only modifications needed are as follows.
When the algorithm produces the (not necessarily primitive) form $ax^2+bxy+cy^2$, say with $(a,b,c)=f\ge1$,
it is counted subject to the following rules, provided that $(f,k)=1$.
\begin{enumerate}
\item Forms proportional to $x^2+y^2$ are counted with weight $1/4$.
\item Forms proportional to $x^2+xy+y^2$ are counted with weight $1/6$.
\item All other forms are counted with weight $1/2$.
\end{enumerate}
Similarly, tables of $M(\G)$ or $M_p(\G)$ values can be computed for $m$ and $k$ of modest size by simultaneously computing a table of values of $H_k(D)$.
\end{rmk}

%%%%%%%%%%%%%%%%%%%%%%%%%%%%%%%%%%%%%%%%%%%%%%%%%%%%%%%%%%%%%%%%%%%%%%%%%%%%%%%%%%%%%%%%%%%%%%%%%%%%%%%%%%%%%%%%%%%%%%%%%%%%%%%%%%%%%%%%%%%%%%%%%%%%%%%%%%%%%%%%%%%%%%%%%%%%%%%%%%%%%%%

\begin{proof}
It follows from~\eqref{reduction to class number avg} that $M_p(G)=0$ unless $p\in(N^-,N^+)$ and $p\equiv 1\pmod m$.
Therefore, assume that $p\in (N^-,N^+)$ and $p\equiv 1\pmod m$, and write $k=s^2t$ with $t$ square-free.
Combining relations~\eqref{inclusion exclusion} and~\eqref{reduction to class number avg}
with the definition of the Kronecker class number, we find that
\begin{equation*}
\begin{split}
M_p(G)
=\sum_{\substack{r\mid s\\ p\equiv 1\pmod{rm}}}\mu(r)H\left(\frac{D_N(p)}{(rm)^2}\right)
&=\sum_{\substack{r\mid s\\ p\equiv 1\pmod{rm}}}\mu(r)H\left(\frac{d(p)}{r^2}\right)\\
&=\sum_{\substack{r\mid s\\ p\equiv 1\pmod{rm}}}\mu(r)
	\sum_{\substack{f^2\mid\frac{d(p)}{r^2}\\ \frac{d(p)}{(rf)^2}\equiv 0,1\pmod 4}}
	\frac{h(d(p)/(rf)^2)}{w(d(p)/(rf)^2)}\\
&=\sum_{\substack{r\mid s\\ p\equiv 1\pmod{rm}}}\mu(r)
	\sum_{\substack{f^2\mid d(p),\, r\mid f\\ \frac{d(p)}{f^2}\equiv 0,1\pmod 4}}\frac{h(d(p)/f^2)}{w(d(p)/f^2)}.
\end{split}
\end{equation*}
Now interchanging the sum over $r$ with the sum over $f$ and recalling the identity
\begin{equation*}
\sum_{r\mid n}\mu(n)=\begin{cases}1&\text{if }n=1,\\ 0&\text{otherwise},\end{cases}
\end{equation*}
we arrive at the formula
\begin{equation*}
\begin{split}
M_p(G)
&=
\sum_{\substack{f^2\mid d(p)\\ (f,s,(p-1)/m)=1\\ \frac{d(p)}{f^2}\equiv 0,1\pmod 4}}\frac{h(d(p)/f^2)}{w(d(p)/f^2)}.
\end{split}
\end{equation*}

In order to complete the proof, it is sufficient to show that, in the above sum,
the condition $(f,s,(p-1)/m)=1$ implies the simpler condition $(f,k)=1$, the converse implication being immediate.
To this end, we write $p=1+jm$ and assume that $(f,s,(p-1)/m)=(f,s,j)=1$.
Then $d(p)=(j-mk)^2-4k$, and the condition $d(p)/f^2\equiv  0, 1\pmod 4$ may be rewritten as
\begin{equation}\label{disc condition}
(j-mk)^2-4k\equiv 0, f^2\pmod{4f^2}.
\end{equation}
Now let $\ell$ be any prime dividing $(f,k)$.
Then the above congruence implies that $\ell\mid j$, but that implies that $\ell^2\mid (j-mk)^2$.
Whence $\ell^2\mid 4k$.
If $\ell$ is odd, then we have that $\ell^2\mid k$, and hence $\ell\mid (f,s,j)=1$, which is a contradiction.
If $\ell=2$, then we divide~\eqref{disc condition} through by $4$ to obtain
\begin{equation*}
\left(\frac{j}{2}-m\frac{k}{2}\right)^2-k\equiv 0, \frac{f^2}{4}\pmod{f^2}.
\end{equation*}
Since $\ell=2\mid (f,k)$,
we have that $k$ is even and congruent to a difference of two squares modulo $4$.
This in turn implies that $k\equiv 0\pmod 4$, i.e., $2\mid s$.
Thus, in this case we also have the contradiction $\ell=2\mid (f,s,j)=1$.
Therefore, we conclude that $(f,k)=1$, and this completes the proof of the lemma.
\end{proof}

%%%%%%%%%%%%%%%%%%%%%%%%%%%%%%%%%%%%%%%%%%%%%%%%%%%%%%%%%%%%%%%%%%%%%%%%%%%%%%%%%%%%%%%%%%%%%%%%%%%%%%%%%%%%%%%%%%%%%%%%%%%%%%%%%%%%%%%%%%%%%%%%%%%%%%%%%%%%%%%%%%%%%%%%%%%%%%%%%%%%%%%

Lemma~\ref{Deuring for groups} together with the class number formula immediately yields Lemma \ref{formula for M(G)}.

%%%%%%%%%%%%%%%%%%%%%%%%%%%%%%%%%%%%%%%%%%%%%%%%%%%%%%%%%%%%%%%%%%%%%%%%%%%%%%%%%%%%%%%%%%%%%%%%%%%%%%%%%%%%%%%%%%%%%%%%%%%%%%%%%%%%%%%%%%%%%%%%%%%%%%%%%%%%%%%%%%%%%%%%%%%%%%%%%%%%%%%%%%%%%%%%%%%%%%%%%%%%%%%%%%%%%%%%%%%%%%%%%%%%%%%%%%%%%%%%%%%%%%%%%%%%%%%%%%%%%%%%%%%%%%%%%%%%%%%%%%%%%%%%%%%%%%%%%%%%%%%%%%%%%%%%%%%%%%%%%%%%%%%%%%%%%%%%%%%%%%%%%%%%%%%%%%%%%%%%%%%%%%%%%%%%

\section{Local computations}\label{local computations}

In this section we gather some local computations which we will need in the proofs of Theorem \ref{approximation of M(G)} and Proposition \ref{startwiththat}.  As before, we continue to assume that $m, k,$ and $N$ are positive integers with $N=|\G|=m^2k$.

%%%%%%%%%%%%%%%%%%%%%%%%%%%%%%%%%%%%%%%%%%%%%%%%%%%%%%%%%%%%%%%%%%%%%%%%%%%%%%%%%%%%%%%%%%%%%%%%%%%%%%%%%%%%%%%%%%%%%%%%%%%%%%%%%%%%%%%%%%%%%%%%%%%%%%%%%%%%%%%%%%%%%%%%%%%%%%%%%%%%%%%

\begin{lma}\label{generic quad lemma}
Let $\ell$ be an odd prime prime. For $e\ge1$, $(d,\ell)=1$ and $(a,b)=1$, we have that
\[
\#\{j\in\Z/\ell^e\Z : j^2\equiv d\pmod{\ell^e}\} =  1+ \leg{d}{\ell}
\]
and
\[
\#\{j\in\Z/\ell^e\Z : j^2\equiv d\pmod{\ell^e},\,(a+bj,\ell)=1\} =  1+ \leg{a^2-db^2}{\ell}^2 \leg{d}{\ell} .
\]
\end{lma}

%%%%%%%%%%%%%%%%%%%%%%%%%%%%%%%%%%%%%%%%%%%%%%%%%%%%%%%%%%%%%%%%%%%%%%%%%%%%%%%%%%%%%%%%%%%%%%%%%%%%%%%%%%%%%%%%%%%%%%%%%%%%%%%%%%%%%%%%%%%%%%%%%%%%%%%%%%%%%%%%%%%%%%%%%%%%%%%%%%%%%%%

\begin{proof} The first formula is classical. For the second, we first note that if $\leg{d}{\ell}=-1$, then $\leg{a^2-db^2}{\ell}^2=1$, and the formula holds. Now assume that $\leg{d}{\ell}=1$, so that there are exactly two solutions to the congruence $j^2\equiv d\mod{\ell^e}$, say $\pm j_0$. If $\ell\mid b$, then the condition $(a+bj,\ell)=1$ is satisfied trivially for all $j\in\Z$, and the claimed result follows. Finally, if $\ell\nmid b$, then we need to exclude exactly one of the solutions when $a\equiv \pm bj_0\mod{\ell}$, that is to say when $a^2\equiv b^2d\mod{\ell}$. So the claimed formula holds in this last  case too.
\end{proof}

%%%%%%%%%%%%%%%%%%%%%%%%%%%%%%%%%%%%%%%%%%%%%%%%%%%%%%%%%%%%%%%%%%%%%%%%%%%%%%%%%%%%%%%%%%%%%%%%%%%%%%%%%%%%%%%%%%%%%%%%%%%%%%%%%%%%%%%%%%%%%%%%%%%%%%%%%%%%%%%%%%%%%%%%%%%%%%%%%%%%%%

We set
\eq{T(n) def}{
T(n)  =  \sum_{d\mod n} \leg{d-4k}{n} \#\{j\mod n: j^2\equiv d\mod n,\ (N+1+jm,n)=1\}.
}

%%%%%%%%%%%%%%%%%%%%%%%%%%%%%%%%%%%%%%%%%%%%%%%%%%%%%%%%%%%%%%%%%%%%%%%%%%%%%%%%%%%%%%%%%%%%%%%%%%%%%%%%%%%%%%%%%%%%%%%%%%%%%%%%%%%%%%%%%%%%%%%%%%%%%%%%%%%%%%%%%%%%%%%%%%%%%%%%%%%%%%%

\begin{prop}\label{odd primes prop} Let $\ell$ be a prime not dividing $2k$ and $w\ge1$. Then
\[
\frac{T(\ell^w)}{\ell^{w-1}}  = - \leg{m(N-1)}{\ell}^2 +
	\begin{cases}
		\ell -  1 - \leg{k}{\ell}
			&\mbox{if $w$ is even},\cr
		 -  1
			&\mbox{if $w$ is odd}.
	\end{cases}
\]
\end{prop}

%%%%%%%%%%%%%%%%%%%%%%%%%%%%%%%%%%%%%%%%%%%%%%%%%%%%%%%%%%%%%%%%%%%%%%%%%%%%%%%%%%%%%%%%%%%%%%%%%%%%%%%%%%%%%%%%%%%%%%%%%%%%%%%%%%%%%%%%%%%%%%%%%%%%%%%%%%%%%%%%%%%%%%%%%%%%%%%%%%%%%%%

\begin{proof} We write $T(\ell^w)=T_1(\ell^w)+T_2(\ell^w)$, where $T_1(\ell^w)$ is the same sum as $T(\ell^w)$ with the additional restriction that $\ell|d$ and $T_2(\ell^w)$ is the remaining sum. First, we calculate $T_1(\ell^w)$. We have that
\als{
T_1(\ell^w)
	&=  \sum_{\substack{d\mod{\ell^w}\\\ell|d}}
	\leg{d-4k}{\ell^w} \sum_{\substack{j\mod{\ell^w} \\ j^2\equiv d\mod {\ell^w}}}\leg{N+1+jm}{\ell}^2 \\
		&= \sum_{\substack{d\mod{\ell^w}\\\ell|d}}
	\leg{-4k}{\ell}^w \leg{N+1}{\ell}^2 \sum_{\substack{j\mod{\ell^w},\, \ell|j \\ j^2\equiv d\mod {\ell^w}}} 1 \\
		&= \leg{-k}{\ell}^w \leg{N+1}{\ell}^2 \sum_{\substack{j\mod{\ell^w} \\ \ell|j }} 1
		= \leg{-k}{\ell}^w \leg{N+1}{\ell}^2 \ell^{w-1} .
}
Finally, we compute $T_2(\ell^w)$. Applying Lemma \ref{generic quad lemma}, we find that
\als{
T_2(\ell^w) &= \sum_{\substack{d\mod{\ell^w}\\ (d,\ell)=1}} \leg{d-4k}{\ell}^w
			\left(  1+\leg{(N+1)^2-dm^2}{\ell}^2 \leg{d}{\ell}  \right)  \\
		&= \ell^{w-1}  \sum_{d\mod{\ell}} \leg{d-4k}{\ell}^w
			\left(  1+\leg{(N+1)^2-dm^2}{\ell}^2 \leg{d}{\ell}  \right)  -  \ell^{w-1}\leg{-k}{\ell}^w.
}
If $\ell\mid m$, then $\leg{(N+1)^2-dm^2}{\ell}=1$ for all $d\mod\ell$. On the other hand, if $\ell \nmid m$, then there is precisely one $d\mod\ell$ such that $(N+1)^2-dm^2\equiv 0\mod{\ell}$, for which we have that
\[
\leg{d-4k}{\ell}^{w}=\leg{m^2d-4m^2k}{\ell}^{w} = \leg{(N-1)^2}{\ell}^{w} = \leg{N-1}{\ell}^2
\quad\text{and}\quad
\leg{d}{\ell} = \leg{N+1}{\ell}^2.
\]
Thus, whether $\ell$ divides $m$ or not, we have
\[
\frac{T_2(\ell^w)}{\ell^{w-1}}   =  - \leg{-k}{\ell}^w  - \leg{m(N-1)(N+1)}{\ell}^2
		+  \sum_{d\mod\ell}   \leg{d-4k}{\ell}^{w} \left( 1+ \leg{d}{\ell}  \right),
\]
which implies that
\als{
\frac{T(\ell^w)}{\ell^{w-1}} &  =  \leg{-k}{\ell}^w \leg{N+1}{\ell}^2  - \leg{-k}{\ell}^w  - \leg{m(N-1)(N+1)}{\ell}^2  \\
		&\quad +  \sum_{d\mod\ell}   \leg{d-4k}{\ell}^{w} \left( 1+ \leg{d}{\ell}  \right) .
}
Note that if $\ell|N+1$, then $\leg{-k}{\ell}=1$ and thus
\[
\leg{-k}{\ell}^w \leg{N+1}{\ell}^2  - \leg{-k}{\ell}^w - \leg{m(N-1)(N+1)}{\ell}^2 = -1 = - \leg{m(N-1)}{\ell}^2,
\]
whereas if $\ell\nmid N+1$, then
\[
\leg{-k}{\ell}^w \leg{N+1}{\ell}^2  - \leg{-k}{\ell}^w - \leg{m(N-1)(N+1)}{\ell}^2 = - \leg{m(N-1)}{\ell}^2 .
\]
So
\als{
\frac{T(\ell^w)}{\ell^{w-1}}
	&=  - \leg{m(N-1)}{\ell}^2
		+  \sum_{d\mod\ell}   \leg{d-4k}{\ell}^{w} \left( 1+ \leg{d}{\ell}  \right) .
}
If now $w$ is odd, then
\[
 \sum_{d\mod\ell}   \leg{d-4k}{\ell}^{w} \left( 1+ \leg{d}{\ell}  \right)
=  \sum_{d\mod\ell}   \leg{d-4k}{\ell}\leg{d}{\ell}
= -1 ,
\]
using for example \cite[Exercise 1.1.9]{Stepanov} since $(2k, \ell)=1$.
Finally, if $w$ is even, then
\[
 \sum_{d\mod\ell}   \leg{d-4k}{\ell}^{w} \left( 1+ \leg{d}{\ell}  \right)
 	=  \ell-1 + \sum_{\substack{d\mod\ell \\ d\not\equiv 4k\mod{\ell}}}  \leg{d}{\ell}
	= \ell-1 - \leg{k}{\ell} ,
\]
which completes the proof of the proposition.
\end{proof}

%%%%%%%%%%%%%%%%%%%%%%%%%%%%%%%%%%%%%%%%%%%%%%%%%%%%%%%%%%%%%%%%%%%%%%%%%%%%%%%%%%%%%%%%%%%%%%%%%%%%%%%%%%%%%%%%%%%%%%%%%%%%%%%%%%%%%%%%%%%%%%%%%%%%%%%%%%%%%%%%%%%%%%%%%%%%%%%%%%%%%%%

\begin{cor}\label{formula for P(ell)} For a prime $\ell$ not dividing $2k$, we have that
\[
P(\ell) := 1 + \sum_{w\ge1} \frac{T(\ell^w)}{\ell^{2w-1}(\ell- \leg{m}{\ell}^2) }
	= \frac{\ell^3-\leg{m}{\ell}^2\ell^2 - (1+\leg{m}{\ell}^2\leg{N-1}{\ell}^2)\ell - 1- \leg{N-1}{\ell}^2\leg{k}{\ell} }
		{(\ell^2-1)(\ell-\leg{m}{\ell}^2)} .
\]
\end{cor}

%%%%%%%%%%%%%%%%%%%%%%%%%%%%%%%%%%%%%%%%%%%%%%%%%%%%%%%%%%%%%%%%%%%%%%%%%%%%%%%%%%%%%%%%%%%%%%%%%%%%%%%%%%%%%%%%%%%%%%%%%%%%%%%%%%%%%%%%%%%%%%%%%%%%%%%%%%%%%%%%%%%%%%%%%%%%%%%%%%%%%%%

\begin{proof}
Lemma \ref{odd primes prop} and a straightforward computation imply that
\[
P(\ell) = \frac{\ell^3-\leg{m}{\ell}^2\ell^2 - (1+\leg{m}{\ell}^2\leg{N-1}{\ell}^2)\ell + \leg{m}{\ell}^2-\leg{m(N-1)}{\ell}^2-1-\leg{k}{\ell} } {(\ell^2-1)(\ell-\leg{m}{\ell}^2)} .
\]
Finally, note that
\[
\leg{m(N-1)}{\ell}^2+\leg{k}{\ell}  - \leg{m}{\ell}^2 =  \leg{N-1}{\ell}^2\leg{k}{\ell} ,
\]
since $\leg{k}{\ell}=\leg{m}{\ell}^2=1$ if $\ell|N-1$.
\end{proof}

%%%%%%%%%%%%%%%%%%%%%%%%%%%%%%%%%%%%%%%%%%%%%%%%%%%%%%%%%%%%%%%%%%%%%%%%%%%%%%%%%%%%%%%%%%%%%%%%%%%%%%%%%%%%%%%%%%%%%%%%%%%%%%%%%%%%%%%%%%%%%%%%%%%%%%%%%%%%%%%%%%%%%%%%%%%%%%%%%%%%%%%%%%%%%%%%%%%%%%%%%%%%%%%%%%%%%%%%%%%%%%%%%%%%%%%%%%%%%%%%%%%%%%%%%%%%%%%%%%%%%%%%%%%%%%%%%%%%%%%%%%%%%%%%%%%%%%%%%%%%%%%%%%%%%%%%%%%%%%%%%%%%%%%%%%%%%%%%%%%%%%%%%%%%%%%%%%%%%%%%%%%%%%%%%%%%

\section{Proof of Proposition \ref{startwiththat}}\label{proof of Prop startwiththat}

%\ccom{new sentence}
This section is dedicated to the proof of Proposition \ref{startwiththat}, which gives an upper bound of the conjectured order of magnitude for the average of special values
\[
\L(d(p)) = L \left(1,  \displaystyle{\left( \frac{d(p)}{\cdot}\right)} \right)
\] 
summed over integers with no small prime factors. A key role will be played by the fundamental lemma of sieve methods, i.e. Lemma \ref{lemma-FI}.

%%%%%%%%%%%%%%%%%%%%%%%%%%%%%%%%%%%%%%%%%%%%%%%%%%%%%%%%%%%%%%%%%%%%%%%%%%%%%%%%%%%%%%%%%%%%%%%%%%%%%%%%%%%%%%%%%%%%%%%%%%%%%%%%%%%%%%%%%%%%%%%%%%%%%%%%%%%%%%%%%%%%%%%%%%%%%%%%%%%%%%%%%%

\begin{proof}[Proof of Proposition \ref{startwiththat}] We shall employ the notation
\[
\rho(n) := \frac{|n|}{\phi(|n|)} = \prod_{\ell | n} \left(1-\frac{1}{\ell}\right)^{-1} .
\]
We will simplify the sum we are estimating with an application of the Cauchy-Schwarz inequality but, first, we massage the
$L$-functions that appear in it. Note that if $p=1+jm$, then $d(p)=(j-mk)^2 - 4k \equiv j^2\pmod k$.
So
\als{
\L(d(p))^r
	= \prod_{ \substack{ \ell | k \\ \ell\nmid j}} \left(1-\frac{1}{\ell} \right)^{-r}
		\prod_{  \ell \nmid k} \left(1-\frac{ \leg{d(p)}{\ell}  }{\ell}\right)^{-r}
	 \ll_r \rho(k)^r \rho((j,k))^{|r|}  \L(k^2d(p))^r,
}
and consequently,
\[
S := \sum_{\substack{N^-<p<N^+ \\ p \equiv 1 \pmod m }}
	 \rho(d(p))^{s} \L(d(p))^r
	\ll_r  \rho(k)^r \sum_{\substack{N^-<p<N^+ \\ p =1+jm,\, j\in\N }} \rho((j,k))^{|r|} \rho(d(p))^s
		 \L(k^2d(p))^r.
\]
Hence the Cauchy-Schwarz inequality yields that
\eq{C-S}{
\frac{S}{\rho(k)^r }
	\ll_r   \left(\sum_{\substack{N^-<p<N^+ \\ p =1+jm }} \rho((j,k))^{2|r|} \rho(d(p))^{2s} \right)^{\frac{1}{2}}
		\left(\sum_{\substack{N^-<p<N^+ \\ p\equiv 1\pmod m }}
			 \L(k^2d(p))^{2r} \right)^{\frac{1}{2}}
	=: \sqrt{S_1 S_2} ,
}
say.

First, we estimate $S_1$. Note that
\[
\rho(n)^v \asymp_v \prod_{\ell|n} \left(1+\frac{v}{\ell}\right)
			= \sum_{a|n} \frac{\mu^2(a) \tau_{v}(a) }{a} ,
\]
for any $v\ge0$. Since
\[
\sum_{\substack{a|n \\ a>x}} \frac{\mu^2(a) \tau_v(a)}{a}
	\le \frac{1}{x}\sum_{a|n} \mu^2(a)v^{\omega(a)}
		= \frac{(v+1)^{\omega(n)}}{x} \ll_{v,\epsilon} \frac{n^\epsilon}{x} ,
\]
we find that
\eq{C-S-e1}{
S_1 &\ll_r  \sum_{\substack{ N^-<p <N^+ \\ p=1+jm}}
		\left( \sum_{\substack{a|(k,j) \\ a\le k^{1/5} }}\frac{\mu^2(a)\tau_{2|r|}(a)}{a} + O_r(k^{-1/6}) \right)
		\left( \sum_{\substack{b|d(p) \\ b\le k^{1/5} }}\frac{\mu^2(b)\tau_{2s}(b)}{b} + O_s(k^{-1/6})\right) \\
	&=   \sum_{ \substack{ a,b\le k^{1/5} \\ a|k }} \frac{\mu^2(a)\mu^2(b)\tau_{2|r|}(a)\tau_{2s}(b) }{ab}
			\sum_{\substack{N^-<p<N^+ \\ p=1+jm \\ a|j,\ b| d(p) }} 1
		 +O_{r,s}( k^{11/30} ),
}
using the trivial estimate $\#\{N^-<p<N^+:p\equiv1\mod m\}\ll \sqrt{N}/m=\sqrt{k}$.
The innermost sum in the second line of \eqref{C-S-e1} equals
\[
\sum_{\substack{ h \in \Z/[a,b]\Z \\ h \equiv 0 \pmod a \\ (h-mk)^2 \equiv 4k \pmod b}}
			 \sum_{\substack{ N^-<p <N^+ \\ p=1+jm \\ j\equiv h \pmod{[a,b]} }} 1
	\ll \frac{\sqrt{N}}{\phi(m[a,b])\log(2k)}
		\sum_{\substack{ h \in \Z/[a,b]\Z \\ h\equiv 0 \pmod a \\ (h-mk)^2 \equiv 4k \pmod b}} 1
	\le \frac{\sqrt{N} \tau(b) }{\phi(m[a,b])\log(2k)} ,
\]
where the first inequality follows from the Brun-Titchmarsch inequality and the second from the fact that $b$ is square-free.
Since $\phi(m[a,b])\ge\phi(m)\phi([a,b])$, relation \eqref{C-S-e1} becomes
\eq{estimation of S_1}{
S_1 &\ll_{r,s} \frac{\sqrt{N}}{\phi(m)\log(2k)} \sum_{ \substack{ a,b\le k^{1/5} \\ a|k }}
		\frac{\mu^2(a)\mu^2(b)\tau_{2|r|}(a)\tau_{2s}(b)^2 }{a\cdot b\cdot \phi([a,b])}
			+ k^{11/30}
	\ll_{r,s} \frac{\sqrt{N}}{\phi(m)\log(2k)} .
}

Next, we turn to the estimation of
\[
S_2 = \sum_{\substack{N^-<p<N^+ \\ p \equiv 1 \pmod m }}  \L(k^2d(p))^{2r} .
\]
Our first task is to replace the $L$-values that appear in the above sum with truncated Euler products. We set
\[
S_3 = \sum_{\substack{N^-<p<N^+ \\ p \equiv 1 \pmod m }}  \L(k^2d(p);z^{80000})^{2r}
\]
with $z=\log(4k)$ and estimate the error
\[
R:=S_2-S_3
\]
using Lemma~\ref{lemmashortproduct}.
First note that since $d(p)$ is a discriminant and $|d(p)|\le 4k$ for $p\in (N^-,N^+)$,
it follows that 
\eq{leg-k^2d(p)}{
\leg{k^2d(p)}{\cdot}
}
is periodic modulo $k|d(p)|\le 4k^2$ and its conductor cannot exceed $|d(p)|\le 4k$.
Thus, we may apply Lemma~\ref{lemmashortproduct} with $\alpha=100$ and $Q=4k$.
Now let $d_1=d_1(p)$ be the discriminant of the quadratic number field $\Q(\sqrt{d(p)})$, so that the character in \eqref{leg-k^2d(p)} is induced by the primitive character $\leg{d_1}{\cdot}$. If $|d_1| \notin \mathcal{E}_{100}(4k)$, then we can approximate $\L(k^2d(p))^{2r}$ very well by $\L(k^2d(p);z^{80000})^{2r}$.
Otherwise, we write $d(p) = d_1 b^2$ and note that
\[
\L(k^2d(p))^{2r}
	\le \rho(kb)^{2|r|} \L(d_1)^{2r}
	\ll_r \rho(kb)^{2|r|}\cdot \begin{cases}
		 (\log|d_1|)^{2r} &\text{if}\ r\ge0,\cr
		 |d_1|^{1/8} &\text{if}\ r<0,
	     \end{cases}
\]
the second estimate being a consequence of Siegel's theorem.
In any case, we find that
\[
\L(k^2d(p))^{2r}
	\ll_r (\rho(kb))^{2|r|} |d_1|^{1/8} \ll_r (kb |d_1|)^{1/8} \le (k |d(p)|)^{1/8} \le (2k)^{1/4}.
\]
Combining the above, we arrive at the estimate
\als{
R &\ll_{r} \sum_{\substack{N^-<p<N^+ \\ p \equiv 1 \pmod m }} \frac{(\log\log k)^{2|r|} }{\log^{100}(2k)}
	+ \sum_{\substack{N^-<p<N^+ \\ p \equiv 1 \pmod m  \\ |d_1| \in\mathcal{E}_{100}(4k)}} k^{1/4}
}
Note that if $p=1+jm$ is such that $|d_1| \in\mathcal{E}_{100}(4k)$, then $d(p)=d_1 b^2$ for some $b\in\mathbb{N}$, or equivalently, $(j-mk)^2-d_1 b^2=4k$.
So for each fixed $d_1$ with $|d_1|\in \mathcal{E}_{100}(4k)$, there are at most $4\tau(4k)\ll k^{1/100}$ admissible values of $j$ (and hence of $p$). Consequently,
\eq{estimation of R}{
R \ll_{r}
	\sum_{\substack{N^-<p<N^+ \\ p \equiv 1 \pmod m }} \frac{(\log\log k)^{2|r|} }{\log^{100}(2k)}
		+ k^{1/4} \cdot k^{1/100} \cdot |\mathcal{E}_{100}(4k)|
  	\ll_{r} \frac{\sqrt{N} }{ \log(2k)\phi(m)},
}
by Lemma\ \ref{lemmashortproduct} and the Brun-Titchmarsch inequality.

Finally, we turn to the estimation of $S_3$. First, note that 
\[
\L(k^2d(p);z^{80000})^{2r}
	\ll_r \L(k^2d(p);\sqrt{z})^{2r}
	\ll_r  \prod_{ \substack{ \ell\nmid 2pk \\ 2|r|+1< \ell \le \sqrt{z} }} \left(1+2r\cdot \frac{\leg{d(p)}{\ell}  }{\ell}\right) ,
\]
by Mertens' estimate, which immediately implies that
\[
S_3 \ll_r
	\sum_{\substack{N^-<p<N^+ \\ p \equiv 1 \pmod{m}  }}
		\prod_{ \substack{ \ell\nmid 2pk \\ 2|r|+1 < \ell \le \sqrt{z} }}
			\left(1+2r\cdot \frac{\leg{d(p)}{\ell}  }{\ell}\right) .
\]
We cannot estimate this sum as it is because that would require information about primes in arithmetic progressions that are currently not available.
We refer the reader to~\cite{DS-MEG} for a more detailed discussion about this issue.
Instead, we extend the summation from primes $p$ to integers $n$ with no prime factors $\le k^{1/8}$ and we apply Lemma~\ref{lemma-FI} with $D=k^{1/4}$ and $y=k^{1/8}$. Hence
\eq{S3-S4}{
S_3 \ll_r \sum_{\substack{N^-< n <N^+\\ n \equiv 1\pmod {m} }}(\lambda^+*1)(n)
	\prod_{ \substack{2|r|+1 < \ell\le\sqrt{z} \\  \ell\nmid 2nk }} \left(1+2r\cdot \frac{\leg{d(n)}{\ell}}{\ell}\right)=:S_4,
}
by the positivity of the above Euler product. Expanding this product to a sum, opening the convolution $(\lambda^+*1)(n)$, and interchanging the order of summation yields
\als{
S_4 &=
	\sum_{\substack{\ell|a\ \Rightarrow\ 2|r|+1<\ell\le \sqrt{z} \\ (a,2k)=1}}\frac{\mu^2(a) \tau_{2r}(a) }{a}
	\sum_{\substack{N^-<n<N^+\\ (n,a)=1 \\ n\equiv 1\pmod{m} }}
		(\lambda^+*1)(n) \leg{d(n)}{a} \nn
  &= 	\sum_{\substack{\ell|a\ \Rightarrow\ 2|r|+1<\ell\le \sqrt{z}  \\ (a,2k)=1}}\frac{\mu^2(a)\tau_{2r}(a) }{a}
	\sum_{\substack{b \leq k^{1/4}\\ (b,am)=1}}\lambda^+(b)
	\sum_{\substack{N^-<n<N^+\\ (n,a)=1,\, b|n \\ n\equiv 1\pmod{m} }}\leg{d(n)}{a} .
}
Splitting the integers $n\in(N^-,N^+)$ according to the congruence class of $d(n)\pmod a$, we deduce that
\eq{S_3}{
S_4 = \sum_{\substack{\ell |a\ \Rightarrow\ 2|r|+1<\ell\le \sqrt{z} \\ (a,2k)=1}}
	\frac{\mu^2(a)\tau_{2r}(a)}{a}
	\sum_{\substack{b \leq k^{1/4}\\ (b,am)=1}}\lambda^+(b)
	\sum_{c\in\Z/a\Z} \leg{c}{a} S(a,b,c),
}
where
\[
S(a,b,c)
	:= \#\left\{ N^-<n<N^+ :
		\begin{array}{ll}
			n\equiv 1\pmod {m} & (n,a)=1 \\
			n\equiv 0\pmod b & d(n)\equiv c\pmod{a}
		\end{array}
	\right\}.
\]
We fix $a$, $b$ and $c$ as above and calculate $S(a,b,c)$. Set $n=1+j m$, and define $\Delta(j)=(j-mk)^2-4k$, so that $d(n)=\Delta(j)$.
Note that $n$ is counted by $S(a,b,c)$ if and only if $mk - 2\sqrt{k} < j < mk+2\sqrt{k}$, $\Delta(j)\equiv c\pmod a$, $1+jm\equiv0\pmod b$ and $(1 + jm, a) = 1$. Thus we have that
\begin{equation}\label{T}
S(a,b,c)=\left(\frac{4\sqrt{k}}{ab}+O(1)\right)J(a,b,c),
\end{equation}
where
\[
J(a,b,c) :=\#\{j\in \Z/ab\Z :
				\Delta(j)\equiv c\pmod a,\
				1 + jm \equiv 0 \pmod b,\
				(1 + jm, a) = 1\}.
\]
By the Chinese remainder theorem, we find that
\[
J(a,b,c)=U(a,c) := \#\{j\in \Z/a\Z: \Delta(j)\equiv c\pmod a,\ (1+ jm,a) = 1\},
\]
since $(b,m)=1$, and thus there is exactly one solution modulo $b$ to the equation $1+jm\equiv 0\pmod b$. Note that $U(a,c)\le \tau(a)$ by Lemma \ref{generic quad lemma} and that
\[
\sum_{c\in\Z/a\Z}\leg{c}{a} U(a,c) = T(a),
\]
where $T(a)$ is defined by relation \eqref{T(n) def}. Together with relations~\eqref{S_3} and~\eqref{T}, this implies that
\als{
S_4	&= 4\sqrt{k} \sum_{\substack{\ell|a\ \Rightarrow\ 2|r|+1<\ell\le \sqrt{z}   \\ (a,2k)=1}}\frac{\mu^2(a)\tau_{2r}(a)T(a)}{a^2}
	\sum_{\substack{b \leq k^{1/4}\\ (b,am)=1}}\frac{\lambda^+(b)}b \\
	&\quad + O \left( k^{1/4} \sum_{ P^+(a)\le\sqrt{z} } \mu^2(a)\tau_{2|r|}(a)\tau(a) \right) .
}
The error term in the above estimate is
\[
\ll k^{1/4} \sum_{ P^+(a)\le\sqrt{z} } \mu^2(a)\tau_{2|r|}(a)\tau(a)
	 = k^{1/4} \prod_{\ell\le \sqrt{z}} (1+4|r|) \ll_r k^{1/3} .
\]
Finally, note that $|T(a)|\le \tau(a)$ for square-free values of $a$, by Proposition \ref{odd primes prop}. So applying Lemma~\ref{lemma-FI} we conclude that
\als{
S_4 &\ll_r  \sqrt{k}	\sum_{\substack{P^+(a)\le\sqrt{z} \\ (a,2k )=1}}\frac{\mu^2(a) \tau(a) \tau_{2|r|}(a)}{a^2}
		\prod_{\substack{\ell\le k^{1/8}\\ \ell\nmid am}}\left(1-\frac1{\ell}\right)+k^{1/3} \\
	&\ll \sqrt{k}\sum_{\substack{P^+(a)\le\sqrt{z} \\ (a,2k )=1}}\frac{\mu^2(a) \tau(a) \tau_{2|r|}(a)}{a^2}
		\frac1{\log(2k)}  \frac{m}{\phi(m)}  \frac{a}{\phi(a)} +k^{1/3} .
}
Inserting this estimate in \eqref{S3-S4}, we obtain the upper bound
\eq{S_3 e2}{
S_3 \ll_r \frac{\sqrt{k}}{\log(2k)}\frac{m}{\phi(m)}\sum_{(a,2k)=1}\frac{\mu^2(a)\tau(a)\tau_{2|r|}(a)}{a\phi(a)}
	\ll_r \frac{\sqrt{k}}{\log(2k)}\frac{m}{\phi(m)}.
}
Combining the above inequality with relations \eqref{C-S}, \eqref{estimation of S_1}, and \eqref{estimation of R} completes the proof of the proposition.
\end{proof}

%%%%%%%%%%%%%%%%%%%%%%%%%%%%%%%%%%%%%%%%%%%%%%%%%%%%%%%%%%%%%%%%%%%%%%%%%%%%%%%%%%%%%%%%%%%%%%%%%%%%%%%%%%%%%%%%%%%%%%%%%%%%%%%%%%%%%%%%%%%%%%%%%%%%%%%%%%%%%%%%%%%%%%%%%%%%%%%%%%%%%%%%%%%%%%%%%%%%%%%%%%%%%%%%%%%%%%%%%%%%%%%%%%%%%%%%%%%%%%%%%%%%%%%%%%%%%%%%%%%%%%%%%%%%%%%%%%%%%%%%%%%%%%%%%%%%%%%%%%%%%%%%%%%%%%%%%%%%%%%%%%%%%%%%%%%%%%%%%%%%%%%%%%%%%%%%%%%%%%%%%%%%%%%%%%%%

\section{Approximating $M(G)$}\label{approx}

In this section, we prove Theorem \ref{approximation of M(G)}. We start with a preliminary lemma.

%%%%%%%%%%%%%%%%%%%%%%%%%%%%%%%%%%%%%%%%%%%%%%%%%%%%%%%%%%%%%%%%%%%%%%%%%%%%%%%%%%%%%%%%%%%%%%%%%%%%%%%%%%%%%%%%%%%%%%%%%%%%%%%%%%%%%%%%%%%%%%%%%%%%%%%%%%%%%%%%%%%%%%%%%%%%%%%%%%%%%%%%%%

\begin{lma}\label{prime sum} Let $N=m^2k>1$ and $d(p)=d_{m,k}(p)$. If $1\le q\le h \le \sqrt{N}$ and $(a,q)=1$, then
\[
\sum_{\substack{N^-<p\le N^+ \\ p\equiv a\mod q }} \sqrt{|d(p)|}
	= \frac{2\pi mk}{\phi(q)\log N}
			+ O\left( \frac{h}{\sqrt{N}} \cdot \frac{mk}{q} + \frac{\sqrt{k}}{h\log N} \int_{N^-}^{N^+} E(y,h;q) dy \right) .
\]
\end{lma}

%%%%%%%%%%%%%%%%%%%%%%%%%%%%%%%%%%%%%%%%%%%%%%%%%%%%%%%%%%%%%%%%%%%%%%%%%%%%%%%%%%%%%%%%%%%%%%%%%%%%%%%%%%%%%%%%%%%%%%%%%%%%%%%%%%%%%%%%%%%%%%%%%%%%%%%%%%%%%%%%%%%%%%%%%%%%%%%%%%%%%%%%%%

\begin{proof}
We note the trivial bound $\#\{t<p\le t+h:p\equiv a\mod q\}\ll h/q$, which we will use several times throughout the proof. We have that
\eq{prime sum log insertion}{
\sum_{\substack{N^-<p\le N^+ \\ p\equiv a\mod q }} \sqrt{|d(p)|}
	= \sum_{\substack{N^-<p\le N^+ \\ p\equiv a\mod q }} \frac{\sqrt{|d(p)|}\log p}{\log N}	
		+ O\left( \frac{\sqrt{k}}{q}  \right) .
}
Note that if $t=N+1+2\sqrt{N}u_0$ and $u_0\in[-1+2\eta,1-\eta]$ with $\eta:=h/\sqrt{4N}$, then 
\als{
\sqrt{|d(t)|} = 2\sqrt{k}\cdot \sqrt{1-u_0^2} 
		&= \frac{2\sqrt{k}}{\eta}\int_{u_0-\eta}^{u_0} \sqrt{1-u^2} \,\dee u 
			+ O\left(\frac{\eta\sqrt{k}}{\sqrt{1-u_0^2}}\right) \\
		&=\frac{4mk}{h}\int_{u_0-\eta}^{u_0} \sqrt{1-u^2} \, \dee u 
			+ O\left(\frac{h\sqrt{k}}{\sqrt{4N-(N+1-t)^2}}\right) .
}
%\fcom{I changed from $|d(t)|$ to $\sqrt {|d(t)|}$ above.}
%Moreover, note that if  then
%\[
%\sqrt{1-u_0^2} = \frac{1}{\eta}\int_{u_0-\eta}^{u_0} \sqrt{1-u^2} du + O\left(\frac{\eta}{\sqrt{1-u_0^2}}\right) .
%\]
%Applying this with $\eta=h/\sqrt{4N}$, we deduce that
Therefore
\als{
\sum_{\substack{N^-<p\le N^+ \\ p\equiv a\mod q }} \frac{\sqrt{|d(p)|}\log p}{\log N} 
	&= \sum_{\substack{10h+ N^-<p\le -10h+N^+\\ p\equiv a\mod q }} \frac{\sqrt{|d(p)|}\log p}{\log N}	
		 + O\left( \frac{h^{1/2}N^{1/4}}{m} \cdot \frac{h}{q}\right) \\
	&= \frac{4mk}{h\log N}\sum_{\substack{N^- +10h<p\le N^+ - 10h \\ p\equiv a\mod q }} (\log p)
		 \int_{\frac{p-N-1-h}{2\sqrt{N}}}^{\frac{p-N-1}{2\sqrt{N}}} 
			\sqrt{1-u^2} \,\dee u  \\
	&\quad  + O\left( \sum_{\substack{N^- +10h<p\le N^+ - 10h \\ p\equiv a\mod q }}
			\frac{h\sqrt{k}}{\sqrt{(N^+-p)(p-N^-)}}
			+ \frac{h^{3/2}N^{1/4}}{mq} \right) \\
	&= \frac{4mk}{h\log N}  \int_{-1+9\eta}^{1-10\eta} \sqrt{1-u^2} 
		\sum_{\substack{N+1+2u\sqrt{N}<p\le N+1+2u\sqrt{N}+h \\ N^- +10h<p\le N^+ - 10h \\ p\equiv a\mod q }} 
			(\log p) \,\dee u  \\
	&\quad  + O\left( \sum_{\substack{N^- +10h<p\le N^+ - 10h \\ p\equiv a\mod q }}
			\frac{h\sqrt{k}}{\sqrt{(N^+-p)(p-N^-)}}
			+ \frac{h^{3/2}N^{1/4}}{mq} \right) .
}
First, we simplify the main term. If $u\in[-1+10\eta,1-11\eta]$, then the condition that $N^- +10h<p\le N^+ - 10h$ can be discarded. On the other hand, if $u\in[-1,1]\setminus[-1+10\eta,1-11\eta]$, then
\als{
 \sqrt{1-u^2} 
		\sum_{\substack{N+1+2u\sqrt{N}<p\le N+1+2u\sqrt{N}+h \\ N^- +10h<p\le N^+ - 10h \\ p\equiv a\mod q }} 
			(\log p) 
	&\le  \sqrt{1-u^2} 
		\sum_{\substack{N+1+2u\sqrt{N}<p\le N+1+2u\sqrt{N}+h \\ p\equiv a\mod q }} 
			(\log p)   \\
	&\ll \sqrt{\eta}\cdot  \frac{h\log N}{q} .
}
Therefore
\als{
&\int_{-1+9\eta}^{1-10\eta} \sqrt{1-u^2} 
		\sum_{\substack{N+1+2u\sqrt{N}<p\le N+1+2u\sqrt{N}+h \\ N^- +10h<p\le N^+ - 10h \\ p\equiv a\mod q }} 
			(\log p) \,\dee u \\
	&\qquad= \int_{-1}^1 \sqrt{1-u^2} 
		\sum_{\substack{N+1+2u\sqrt{N}<p\le N+1+2u\sqrt{N}+h \\ p\equiv a\mod q }} 
			(\log p) \, \dee u
		 + O\left( \frac{\eta^{3/2} h\log N}{q} \right) \\
	&\qquad= \int_{-1}^{1} \sqrt{1-u^2} \frac{h}{\phi(q)} \, \dee u  
		 + O\left(  \int_{-1}^1 E(N+1+2u\sqrt{N},h;q) \dee u  + \frac{h^{5/2}\log N}{N^{3/4} q} \right) \\
	&\qquad=\frac{\pi}{2}\cdot \frac{h}{\phi(q)} + 
		O\left( \frac{1}{\sqrt{N}} \int_{N^-}^{N^+} E(y,h;q) \dee y  + \frac{h^{5/2} \log N}{N^{3/4} q} \right) .
}
Consequently,
\als{
\sum_{\substack{N^-<p\le N^+ \\ p\equiv a\mod q }} \sqrt{|d(p)|}
	&= \frac{2\pi mk}{\phi(q)\log N}
	 + O\left( \sum_{\substack{N^- +10h<p\le N^+ - 10h \\ p\equiv a\mod q }} \frac{h\sqrt{k}}{\sqrt{(N^+-p)(p-N^-)}}\right) \\
	 &+O\left( \frac{\sqrt{k}}{h\log N} \int_{N^-}^{N^+} E(y,h;q) dy  + \frac{\sqrt{k}}{q} + \frac{h^{3/2}N^{1/4}}{mq}\right) ,
}
where the term $\sqrt{k}/q$ inside the big-Oh comes from \eqref{prime sum log insertion}.
%\fcom{Should $\frac{\sqrt{k}}{q}$ be $\frac{h\sqrt{k}}{q}$?}\dcom{This error term comes from the first displayed equation of the proof (when we add the $\log p$ weight), so I think it's $\sqrt{k}/q$.} \fcom{Could you explain what contribution from $\frac{\eta^{3/2} h\log N}{q}$ is? I keep getting $\frac{h\sqrt k}{q}$ from this term, but I probably misunderstand something. }
%
It remains to bound 
\[
\sum_{\substack{N^- +10h<p\le N^+ - 10h \\ p\equiv a\mod q }} \frac{1}{\sqrt{(N^+-p)(p-N^-)}} .
\]
We break this sum into two pieces, according to whether $p\le N+1$ or $p>N+1$. Note that  %\fcom{Dimitris, could you explain why we could bound the above by $$ \sum_{\substack{N^- +10h<p\le N+1  \\ p\equiv a\mod q }} \frac{1}{\sqrt{(N^+-p)(p-N^-)}}? $$}\dcom{We cannot, we just break the sum into two pieces and deal with them separately.}
\als{
\sum_{\substack{N^- +10h<p\le N+1  \\ p\equiv a\mod q }} \frac{1}{\sqrt{(N^+-p)(p-N^-)}} 
	&\ll N^{-1/4}\sum_{\substack{N^- +10h<n\le N+1  \\ n\equiv a\mod q }} \frac{1}{\sqrt{n-N^-}} .
}
We cover the range of summation by intervals of length $h$ to find that
\als{
\sum_{\substack{N^- +10h<p\le N+1  \\ p\equiv a\mod q }} \frac{1}{\sqrt{(N^+-p)(p-N^-)}} 
	&\ll N^{-1/4} \sum_{1\le j\le 2\sqrt{N}/h} \frac{1}{\sqrt{jh}} \cdot 
		\sum_{\substack{N^-+jh<n\le N^-+jh+h \\ n\equiv a\mod q }} 1 \\
	&\ll \frac{\sqrt{h}}{N^{1/4}q} \sum_{1\le j\le 2\sqrt{N}/h} \frac{1}{\sqrt{j}} 
		\ll \frac{1}{q} .
}
%\fcom{I think the sum $\sum_{N^-+jh<j\le (N+1)/h+1} \frac{1}{\sqrt{j}} $ is probably $\sum_{1\leq j\le \frac{2\sqrt N}{h}} \frac{1}{\sqrt{j}}$? Also is the below equation the same as the equation above?} \dcom{Corrected.} 
Similarly, we find that
\[
\sum_{\substack{N+1<p\le N^+ -10h  \\ p\equiv a\mod q }} \frac{1}{\sqrt{(N^+-p)(p-N^-)}} \ll \frac{1}{q} 
\]
too, which implies that
\als{
\sum_{\substack{N^-<p\le N^+ \\ p\equiv a\mod q }} \sqrt{|d(p)|}
	&= \frac{2\pi mk}{\phi(q)\log N}
	  + O\left( \frac{\sqrt{k}}{h\log N} \int_{N^-}^{N^+} E(y,h;q) dy  + \frac{h\sqrt{k}}{q} + \frac{h^{3/2}N^{1/4}}{mq}\right) .
}
Since $h^{3/2}=N^{3/4} (h/\sqrt{N})^{3/2}\le N^{3/4} (h/\sqrt{N})$, the lemma follows.
\end{proof}

%%%%%%%%%%%%%%%%%%%%%%%%%%%%%%%%%%%%%%%%%%%%%%%%%%%%%%%%%%%%%%%%%%%%%%%%%%%%%%%%%%%%%%%%%%%%%%%%%%%%%%%%%%%%%%%%%%%%%%%%%%%%%%%%%%%%%%%%%%%%%%%%%%%%%%%%%%%%%%%%%%%%%%%%%%%%%%%%%%%%%%%%%%

Using the above result and the results of Section \ref{local computations}, we will prove Theorem \ref{approximation of M(G)}. But first, we need to introduce some additional notation and state another intermediate result. Set
\eq{J_r(v) def}{
J_r(v)  = \{1\le j\le 2^{2v+3}:(j-mk)^2\equiv 4k+4^vr\pmod{2^{2v+3}}, \,   jm\equiv 0\pmod 2\}
}
and
\eq{J(v)}{
\J(v)  =   \frac{1}{2^{v_0-1}} \sum_{r\in\{0,1,4,5\}}  \frac{|J_r(v)|}{2-\leg{r}{2}},
\quad\text{where}\quad
v_0  =
	\begin{cases}
		2 	&\text{if}\ 2\nmid m,\cr
		3 	&\text{if}\ 2|m.
	\end{cases}
}
Finally, set
\[
\J = \sum_{\substack{v\ge0 \\ (2^v,k)=1}} \frac{\J(v)}{8^v} .
\]
Then we have the following formula.

%%%%%%%%%%%%%%%%%%%%%%%%%%%%%%%%%%%%%%%%%%%%%%%%%%%%%%%%%%%%%%%%%%%%%%%%%%%%%%%%%%%%%%%%%%%%%%%%%%%%%%%%%%%%%%%%%%%%%%%%%%%%%%%%%%%%%%%%%%%%%%%%%%%%%%%%%%%%%%%%%%%%%%%%%%%%%%%%%%%%%%%%%%

\begin{lma}\label{formula for J}
\[
\J =
	\begin{cases}
		\frac{2}{3} 	&\mbox{if $2\nmid mk$},\cr
		\frac{3}{2} 		&\mbox{if $2\mid (m,k)$},\cr
		1	&\mbox{if $2\mid mk$, $2\nmid (m,k)$}.
	\end{cases}
\]
\end{lma}

%%%%%%%%%%%%%%%%%%%%%%%%%%%%%%%%%%%%%%%%%%%%%%%%%%%%%%%%%%%%%%%%%%%%%%%%%%%%%%%%%%%%%%%%%%%%%%%%%%%%%%%%%%%%%%%%%%%%%%%%%%%%%%%%%%%%%%%%%%%%%%%%%%%%%%%%%%%%%%%%%%%%%%%%%%%%%%%%%%%%%%%%%%

We postpone the proof of this lemma till the last section.

\begin{proof}[Proof of Theorem \ref{approximation of M(G)}] We will show the theorem with $8\epsilon\in(0,1/3]$ in place of $\epsilon$ and when $k$ is large enough in terms of $\epsilon$, which is clearly sufficient. Our starting point is Lemma \ref{formula for M(G)}, which states that
\[
M(G)
	=\sum_{\substack{N^-<p<N^+ \\ p \equiv 1 \pmod m }}
		\sum_{\substack{f^2\mid d(p), \, (f,k)=1  \\ d(p)/f^2\equiv 1,0\pmod 4  }}
	\frac{\sqrt{|d(p)|} \L(d(p)/f^2)  }{2\pi f}  ,
\]
where $N=m^2k$ and $d(p)=d_{m,k}(p)=((p-N-1)^2-4N)/m^2$ as usual. If $p=1+jm$, then $d(p)=(j-mk)^2-4k$.
Therefore, if $\ell$ is an odd prime dividing $k$, so that $(\ell,f)=1$ for $f$ as in the above sum, then
\[
\leg{d(p)/f^2}{\ell}
=\leg{d(p)}{\ell} = \leg{j}{\ell}^2 ,
\]
Next, we write $f=2^vg$ with $g$ odd and consider $r\in\{0,1,4,5\}$ such that $d(p)/f^2\equiv r\mod 8$.
Then we have that $\leg{d(p)/f^2}{2}=\leg{r}{2}$. Moreover, since $g^2\equiv1\mod{8}$, we have that
\[
d(p)/f^2
\equiv  d(p)/2^{2v}\mod{8},
\]
Therefore, the conditions $f^2|d(p)$ and $d(p)/f^2\equiv r\mod 8$ are equivalent to having $d(p)\equiv 4^v r\mod{2^{2v+3}}$ and $g^2|d(p)$. Setting
\[
\rho(g,d) = \prod_{\ell|g} \left(1-\frac{\leg{d}{\ell}}{\ell} \right)^{-1}
\]
then gives us that
\[
\L(d(p)/f^2)
	= \L((2kg)^2d(p))  \frac{\rho(g,d(p)/g^2)}{1-\leg{r}{2}/2} \prod_{\ell|k,\,\ell\nmid 2j}\left(1-\frac{1}{\ell}\right)^{-1}.
\]
Since
\[
\prod_{\ell|k,\,\ell\nmid 2j}\left(1-\frac{1}{\ell}\right)^{-1}
	= \sum_{\substack{a|k \\ (a,2j)=1}} \frac{\mu^2(a)}{\phi(a)},
\]
we deduce that
\als{
M(G)
	&= \sum_{r\in\{0,1,4,5\}} \frac{1}{2-\leg{r}{2}} \sum_{\substack{N^-<p<N^+ \\ p \equiv 1 \pmod m }}
		\sum_{\substack{a|k \\ (a,2j)=1}}
		\sum_{\substack{v\ge 0,\, (2^v,k)=1  \\ d(p)\equiv 4^v r\mod{2^{2v+3}} }}
		\sum_{\substack{g^2\mid d(p)\\ (g,2k)=1 }}  \frac{\mu^2(a)  \sqrt{|d(p)|}}{\pi 2^v\phi(a)g}  \\
	&\qquad \times 	\rho(g,d(p)/g^2) \L((2kg)^2d(p) )   .
}
We now use Lemma~\ref{lemmashortproduct} to replace the $L$-value $\L((2kg)^2d(p))$ by a suitably truncated product.
Arguing as in the proof of relation~\eqref{estimation of R}, we note that $\leg{(2kg)^2d(p)}{\cdot}$ is a character modulo
$2kg|d(p)|\le 16k^{5/2}$ with conductor not exceeding $|d(p)|\le 4k$.
Thus, we may apply Lemma~\ref{lemmashortproduct} with $Q=4k$ and $5\alpha$ in place of $\alpha$ to replace $\L((2kg)^2d(p))$ by $\L((2kg)^2d(p); z)$,
where we take $z=(\log(4k))^{200\alpha^2}$.
The result is that
\als{
M(G)
	&= \sum_{r\in\{0,1,4,5\}} \frac{1}{2-\leg{r}{2}} \sum_{\substack{N^-<p<N^+ \\ p=1+jm,\,j\ge1 }}
		\sum_{\substack{a|k \\ (a,2j)=1}}
		\sum_{\substack{ (2^v,k)=1  \\ d(p)\equiv 4^v r\mod{2^{2v+3}} }}
		\sum_{\substack{g^2\mid d(p) \\ (g,2k)=1 }}  \frac{\mu^2(a) \sqrt{|d(p)|}}{\pi 2^v\phi(a) g}  \\
	&\qquad \times  \rho(g,d(p)/g^2) \L((2kg)^2d(p) ; z )
		+O_{\alpha}\left(\frac{k}{(\log k)^{\alpha}} \right) .
}
Next, we notice that we can truncate the sums over $a,g$ and $v$ at the cost of a small error term.
More precisely, using the crude bound
\[
 	\rho(g,d(p)/g^2)\L((2kg)^2d(p);z) \ll \frac{g}{\phi(g)} \log(2kg|d(p)|)\ll (\log k)^2 ,
\]
we find that the contribution to $M(G)$ by those summands with $\max\{a,g,2^v\}>k^\epsilon$ is
\eq{divisor bound}{
\ll \frac{\sqrt{k} (\log k)^3}{k^\epsilon} \sum_{\substack{N^-<p<N^+ \\ p \equiv 1 \pmod m }}
		\sum_{\substack{a|k \\ (2^vg)^2|d(p)}} 1
	\ll_\epsilon k^{(1-\epsilon)/2} \sum_{\substack{ N^-<n<N^+ \\  n\equiv 1\mod m }} 1
	 \ll k^{1-\epsilon/2}
}
by the bound $\tau(n)\ll_\delta n^\delta$, with $\delta < \epsilon/4$. Moreover,
\[
	\L((2kg)^2d(p) ; z)
	= \sum_{\substack{P^+(n)\le z \\ (n,2kg)=1}} \frac{\leg{d(p)}{n} }{n}
	= \sum_{\substack{P^+(n)\le z,\, n\le k^\epsilon \\ (n,2kg)=1}} \frac{ \leg{d(p)}{n} }{n}
		+ O_{\epsilon,\alpha}\left( (\log k)^{-\alpha-10}\right)
\]
by Lemma \ref{smooth}. Therefore,
\als{
M(G) &= \sum_{r\in\{0,1,4,5\}} \frac{1}{2-\leg{r}{2}} \sum_{\substack{a|k,\, a\le k^\epsilon \\ (a,2)=1}}
		\sum_{\substack{2^v\le k^\epsilon \\ (2^v,k)=1}}
		\sum_{\substack{g\le k^\epsilon \\ (g,2k)=1 }}
			 \sum_{\substack{P^+(n)\le z,\, n\le k^\epsilon \\ (n,2kg)=1}} \frac{\mu^2(a)}{\pi 2^v\phi(a) g n}  \\
		&\qquad \times
			\sum_{\substack{N^-<p<N^+ \\ p=1+jm,\,j\ge1 \\ (a,j)=1,\, g^2|d(p) \\ d(p)\equiv 4^v r \mod{2^{2v+3}} }}
				\rho(g,d(p)/g^2) \leg{d(p)}{n}\sqrt{|d(p)|}
		 	 + O_{\alpha,\epsilon} \left(\frac{k}{(\log k)^{\alpha}}\right) .
}
We note that if $d(p)/g^2\equiv b\mod{g}$, then $\leg{d(p)/g^2}{\ell} = \leg{b}{\ell}$ for all $\ell|g$ and consequently,
$\rho(g,d(p)/g^2)=\rho(g,b)$. So summing over possible choices for $d(p)/g^2\mod g$ and $d(p)\mod n$, we deduce that
\als{
M(G) &= \sum_{r\in\{0,1,4,5\}} \frac{1}{2-\leg{r}{2}} \sum_{\substack{a|k,\, a\le k^\epsilon \\ (a,2)=1}}
		\sum_{\substack{2^v\le k^\epsilon \\ (2^v,k)=1}}
		\sum_{\substack{g\le k^\epsilon \\ (g,2k)=1 }}
			\sum_{\substack{P^+(n)\le z,\, n\le k^\epsilon \\ (n,2kg)=1}} \frac{\mu^2(a)}{\pi 2^v \phi(a) g n }   \\
		&\qquad \times	\sum_{b=1}^g \rho(g,b)
			\sum_{c=1}^n\leg{c}{n}
		 S_r(v,a,g,b,n,c) + O_{\alpha,\epsilon} \left(\frac{k}{(\log k)^{\alpha}}\right) ,
}
where
\[
S_r(v,a,g,b,n,c) : = \sum_{\substack{N^-<p\le N^+ \\ p=1+jm,\,j\ge 1,\, (j,a)=1 \\ d(p)\equiv bg^2\mod{g^3} \\
			d(p)\equiv 4^vr \mod{2^{2v+3}},\, d(p)\equiv c\mod{n}  }}  \sqrt{|d(p)|} .
\]
We write $p=1+jm$ and note that $(1+jm,2agn)=1$ if $k$ is large enough, since $2agn\le 2k^{3\epsilon}\le 2k^{1/8}$ by assumption, and $p>N^-=(m\sqrt{k}-1)^2$. Moreover, with this notation we have that $d(p)=\Delta(j):=(j-mk)^2-4k$. So, if we set
\als{
J_r(v,a,g,b,n,c) =
	\left\{ j\mod{2^{2v+3} a g^3n}:
		\begin{array}{rl}
		 \Delta(j)  \equiv 4^v r  \pmod{2^{2v+3} },& \Delta(j) \equiv b g^2 \pmod{g^3}, \\
		 \Delta(j)\equiv c\mod{n},& (j,a)=1,\\
		(1+jm, agn)=1,& jm\equiv 0\pmod2
		\end{array}\right\},
}
then we find that
\[
S_r(v,a,g,b,n,c) = \sum_{j\in J_r(v,a,g,b,n,c)}
	 \sum_{\substack{N^-<p\le N^+ \\ p\equiv 1+jm\mod{2^{2v+3}ag^3nm}}}
		\sqrt{|d(p)|} .
\]
Applying Lemma \ref{prime sum} with $h$ as in the statement of the theorem, we deduce that
\als{
\frac{S_r(v,a,g,b,n,c)}{ |J_r(v,a,g,b,n,c)|}
	&= \frac{2\pi mk}{\phi(2^{2v+3}ag^3nm)\log N} \\
	&\quad	+ O\left( \frac{k}{4^vag^3n(\log k)^{\alpha+1}}
				+ \frac{\sqrt{k}}{h\log k} \int_{N^-}^{N^+} E(y,h;2^{2v+3}ag^3nm) \dee y \right) ,
}
by our assumption that $h\le m\sqrt{k}/(\log k)^{\alpha+1}$ and that $m\le \sqrt{k}$. In order to compute the contribution of the above error term to $M(G)$, we note that
\als{
\sum_{b=1}^g \rho(b,g) \sum_{c=1}^n |J_r(b,v,g,a,n,c)|
	&\le \sum_{b=1}^g \sum_{c=1}^n
		\sum_{ \substack{  j\mod {2^{2v+3} ag^3n} \\ \Delta(j)\equiv bg^2\pmod{g^3} \\
		 \Delta(j)\equiv 4^vr \pmod{2^{2v+3}} \\ 2|jm,\ \Delta(j)\equiv c \mod{n} }} \frac{g}{\phi(g)}
		=  \sum_{ \substack{  j\mod {2^{2v+3} a g^3n} \\ g^2 | \Delta(j),\, 2|jm \\   \Delta(j)\equiv 4^vr\pmod{2^{2v+3}} }} 					\frac{g}{\phi(g)}  \\
	&= \frac{g}{\phi(g)} agn \sum_{ \substack{  j\mod {2^{2v+3} g^2} \\ 2|jm,\ g^2 | \Delta(j) \\   \Delta(j)\equiv 4^vr\pmod{2^{2v+3}} }} 1
		\ll \frac{ag^2n}{\phi(g)} \cdot \tau(g) \cdot   |J_r(v)|
}
by the Chinese remainder theorem and Lemma \ref{generic quad lemma}, where $J_r(v)$ is defined by \eqref{J_r(v) def}. Since we also have that $|J_r(v)|\ll \J(v) \ll1$ by Lemmas \ref{prime 2 lma1} and \ref{prime 2 lma2} below, we conclude that
\als{
M(G) &= \frac{2mk}{\log N}\sum_{r\in\{0,1,4,5\}} \frac{1}{2-\leg{r}{2}}
		\sum_{\substack{a|k,\, a\le k^\epsilon \\ (a,2)=1}}
		\sum_{\substack{2^v\le k^\epsilon \\ (2^v,k)=1}}
		\sum_{\substack{g\le k^\epsilon \\ (g,2k)=1 }}
			\sum_{\substack{P^+(n)\le z,\, n\le k^\epsilon \\ (n,2kg)=1}}
			 \frac{\mu^2(a)}{2^{3v+v_0}\phi(a)\phi(g^4an^2m)} \\
	&\qquad\times \sum_{b=1}^g \rho(g,b)
	 \sum_{c=1}^n \leg{c}{n}|J_r(b,v,g,a,n,c)|
		+ O_{\alpha,\epsilon}\left( \frac{k}{(\log k)^\alpha} + E\right),
}
where $v_0$ is defined by \eqref{J(v)} and
\[
E:= \frac{\sqrt{k}}{h} \sum_{q\le 8k^{7\epsilon}} \tau_3(q) \int_{N^-}^{N^+} E(y,h;mq) \dee y ,
\]
since, for any $q\in\N$, we have that
\[
\sum_{\substack{q=2^{2v+3}ag^3n \\ a|k,\ (a,2)=(gn,2k)=1}} \tau(g) \le \sum_{g|q} \tau(g) = \tau_3(q) .
\]
If we set
\[
I(g,b) = \#\{1\le j\le g^3: \Delta(j)\equiv bg^2\pmod{g^3}, \,  (1+jm,g)=1\}
\]
and
\[
F(a) = \#\{ 1\le j\le a: (j,a)=1,\,(1+jm,a)=1\} = \prod_{\ell^w \| a} \ell^{w-1} \left(\ell-1-\leg{m}{\ell}^2\right) ,
\]
then the Chinese remainder theorem implies that
\als{
\sum_{c=1}^n \leg{c}{n} |J_r(v,a,g,b,n,c)|
	&= F(a) \cdot |J_r(v)| \cdot  I(g,b) \sum_{c=1}^n \leg{c}{n} 	
			\sum_{\substack{ j\mod n \\ \Delta(j)\equiv c\mod{n} \\(1+jm,n)=1 }}1 \\
	&= F(a) \cdot |J_r(v)|\cdot  I(g,b) \cdot T(n),
}
where $T(n)$ is defined by \eqref{T(n) def}. Therefore,
\[
M(G) = \frac{mk}{\phi(m)\log N} S_1S_2S_3+ O_{\alpha,\epsilon}\left( \frac{k}{(\log k)^\alpha} + E\right),
\]
where
\[
S_1 = \sum_{r\in\{0,1,4,5\}} \frac{2}{2-\leg{r}{2}}
		\sum_{\substack{2^v\le k^\epsilon \\ (2^v,k)=1}} \frac{|J_r(v)|}{2^{3v+v_0}}
	=\J  + O(k^{-\epsilon}),
\]
by the trivial estimate $|J_r(v)|\ll 4^v$,
\als{
S_2 = \sum_{\substack{a|k,\,a\le k^\epsilon \\(a,2)=1}} \frac{\mu^2(a) F(a)}{\phi(a)a}
		\prod_{\ell|a,\,\ell\nmid m} \frac{\ell}{\ell-1}
	&= \prod_{\substack{\ell|k\\ \ell\ne 2}} \left( 1+ \frac{\ell-1-\leg{m}{\ell}^2}{(\ell-1)(\ell-\leg{m}{\ell}^2)}\right)
		+ O(k^{-\epsilon/2}) \\
	&= \prod_{\substack{\ell|k\\ \ell\ne 2}} \frac{\ell^2- \leg{m}{\ell}^2 \ell -1 }{(\ell-1)(\ell-\leg{m}{\ell}^2)} + O(k^{-\epsilon/2})
}
by arguing as in relation \eqref{divisor bound}, and
\[
S_3 = \sum_{\substack{g\le k^\epsilon \\ (g,2k)=1 }} \sum_{b=1}^g \frac{\rho(g,b) I(g,b) S_4(g) }{g^4}
		 \prod_{\ell|g,\,\ell\nmid m} \frac{\ell}{\ell-1}
\]
with
\[
S_4(g) = \sum_{\substack{P^+(n)\le z,\, n\le k^\epsilon \\ (n,2kg)=1}} \frac{T(n)}{n^2}\prod_{\ell|n,\,\ell\nmid m}\frac{\ell}{\ell-1} .
\]
In the above,  to factor $\phi(g^4 a n^2 m)$, we have used the identity
$$
\phi(g^4 a n^2 m) = \phi(m) g^4 a n^2  \prod_{\ell|g,\,\ell\nmid m} \frac{\ell -1}{\ell} \;
 \prod_{\ell|a,\,\ell\nmid m} \frac{\ell - 1}{\ell} \;
 \prod_{\ell|n,\,\ell\nmid m} \frac{\ell - 1}{\ell}
$$
which holds since $a,n$ and $g$ are pairwise coprime. Note that
\als{
I(g,b) &= \prod_{\ell^w\| g} \#\{j\mod {\ell^{3w}}: (j-mk)^2\equiv 4k+bg^2\mod{\ell^{3w}},\, (1+jm,\ell)=1\}  \\
		& = \prod_{\ell|g} \left(1+ \leg{(N+1)^2-(4k+bg^2)m^2}{\ell}^2  \leg{4k+bg^2}{\ell}\right) \\
		& =\left(1+ \leg{N-1}{\ell}^2\leg{k}{\ell} \right)^{\omega(g)}
}
by Lemma \ref{generic quad lemma}, which is applicable here because $4k+bg^2\equiv 4k\not\equiv 0\mod\ell$ for all primes $\ell|g$. So we see that $I(g,b)$ is independent of $b$, which implies that
\als{
\sum_{b=1}^g \rho(g,b) I(g,b)
	&= I(g,0)  \prod_{\ell^w\|g} \left( \sum_{b=1}^{\ell^w} \frac{1}{1-\leg{b}{\ell}/\ell} \right) \\
	&= I(g,0)  \prod_{\ell^w\|g} \left(\ell^{w-1}+ \ell^{w-1}\frac{\ell-1}{2}\frac{1}{1-1/\ell}
		+\ell^{w-1}\frac{\ell-1}{2}\frac{1}{1+1/\ell}\right) \\
	&= g I(g,0) \prod_{\ell|g} \frac{\ell^2+\ell+1}{\ell(\ell+1)}  .
}
Thus we conclude that
\[
S_3 = \sum_{\substack{g\le k^\epsilon \\ (g,2k)=1 }}
		 \frac{S_4(g)}{g^3} \prod_{\ell|g}  \frac{(1+\leg{N-1}{\ell}^2\leg{k}{\ell})(\ell^2+\ell+1)}{(\ell-\leg{m}{\ell}^2)(\ell+1)}.
\]
Moreover, if $P(\ell)$ is as in Corollary \ref{formula for P(ell)}, then we have that
\[
S_4(g) = \frac{P}{\prod_{\ell|g} P(\ell)} \left( 1
	+ O\left(\frac{1}{(\log k)^{\alpha+1}}\right) \right) ,
\quad\text{where}\quad
P : = \prod_{\ell\nmid 2k} P(\ell) .
\]
Therefore
\als{
S_3\left( 1+ O\left(\frac{1}{(\log k)^{\alpha+1}}\right) \right)
	&= P \cdot \prod_{\ell\nmid2k} \left( 1 + \sum_{w\ge1} \frac{(1+\leg{N-1}{\ell}^2\leg{k}{\ell})(\ell^2+\ell+1)}{\ell^{3w}(\ell-\leg{m}{\ell}^2)(\ell+1)P(\ell)}\right) \\
	&= \prod_{\ell\nmid2k} \left( P(\ell) +  \frac{1+\leg{N-1}{\ell}^2\leg{k}{\ell}}{(\ell^{2}-1)(\ell-\leg{m}{\ell}^2)} \right) \\
	&= \prod_{\ell\nmid2k}  \frac{\ell^3 - \leg{m}{\ell}^2 \ell^2 - (1+\leg{m(N-1)}{\ell}^2)\ell }{(\ell^{2}-1)(\ell-\leg{m}{\ell}^2)}\\
	&= \prod_{\ell\nmid2N}  \left( 1- \frac{\ell \leg{N-1}{\ell}^2+1}{(\ell^{2}-1)(\ell-1)}  \right)   .
}
Consequently,
\als{
M(G) &= \frac{\J mk}{\phi(m)\log N}  \prod_{\ell\nmid2N}  \left( 1- \frac{\ell \leg{N-1}{\ell}^2+1}{(\ell^{2}-1)(\ell-1)}  \right)
		 \prod_{\substack{\ell|k \\ \ell>2}} \left( 1+ \frac{\ell-1-\leg{m}{\ell}^2}{(\ell-1)(\ell-\leg{m}{\ell}^2)}\right)   \\
		&\quad + O_{\alpha,\epsilon}\left( \frac{k}{(\log k)^\alpha} + E\right) .
}
So the theorem follows by the above estimates together with Lemmas \ref{Aut(G)} and \ref{formula for J}.
\end{proof}

%%%%%%%%%%%%%%%%%%%%%%%%%%%%%%%%%%%%%%%%%%%%%%%%%%%%%%%%%%%%%%%%%%%%%%%%%%%%%%%%%%%%%%%%%%%%%%%%%%%%%%%%%%%%%%%%%%%%%%%%%%%%%%%%%%%%%%%%%%%%%%%%%%%%%%%%%%%%%%%%%%%%%%%%%%%%%%%%%%%%%%%%%%%%%%%%%%%%%%%%%%%%%%%%%%%%%%%%%%%%%%%%%%%%%%%%%%%%%%%%%%%%%%%%%%%%%%%%%%%%%%%%%%%%%%%%%%%%%%%%%%%%%%%%%%%%%%%%%%%%%%%%%%%%%%%%%%%%%%%%%%%%%%%%%%%%%%%%%%%%%%%%%%%%%%%%%%%%%%%%%%%%%%%%%%%%

\section{Powers of 2}\label{2}

The goal of this section is to show Lemma \ref{formula for J}, which gives the value of
\begin{equation*}
\J = \sum_{\substack{v\ge0 \\ (2^v,k)=1}} \frac{\J(v)}{8^v},
\end{equation*}
where
\begin{equation*}
\J(v)  =   \frac{1}{2^{v_0-1}} \sum_{r\in\{0,1,4,5\}}  \frac{|J_r(v)|}{2-\leg{r}{2}},
\quad\quad
v_0  =
	\begin{cases}
		2 	&\text{if}\ 2\nmid m,\cr
		3 	&\text{if}\ 2|m,
	\end{cases}
\end{equation*}
and
\begin{equation*}
J_r(v)  = \{1\le j\le 2^{2v+3}:(j-mk)^2\equiv 4k+4^vr\pmod{2^{2v+3}}, \,   jm\equiv 0\pmod 2\}.
\end{equation*}
We start with the following standard lemma.

%%%%%%%%%%%%%%%%%%%%%%%%%%%%%%%%%%%%%%%%%%%%%%%%%%%%%%%%%%%%%%%%%%%%%%%%%%%%%%%%%%%%%%%%%%%%%%%%%%%%%%%%%%%%%%%%%%%%%%%%%%%%%%%%%%%%%%%%%%%%%%%%%%%%%%%%%%%%%%%%%%%%%%%%%%%%%%%%%%%%%%%%%%

\begin{lma}\label{generic quad lemma - prime 2}
We have that
\[
\#\{j\in\Z/8\Z : j^2\equiv d\pmod{8}\} =
	\begin{cases}
		2 &\text{if}\ d\equiv 0,4\mod 8,\\
		4 &\text{if}\ d\equiv 1\mod 8,\\
		0 &\text{otherwise}.
	\end{cases}
\]
Moreover, if $d$ is odd and $e\ge3$, then
\[
\#\{j\in\Z/2^e\Z : j^2\equiv d\pmod{2^e}\} =
	\begin{cases}
		4 	&\text{if}\ d\equiv 1\mod 8, \\
		0	&\text{otherwise}.
	\end{cases}
\]
\end{lma}

%%%%%%%%%%%%%%%%%%%%%%%%%%%%%%%%%%%%%%%%%%%%%%%%%%%%%%%%%%%%%%%%%%%%%%%%%%%%%%%%%%%%%%%%%%%%%%%%%%%%%%%%%%%%%%%%%%%%%%%%%%%%%%%%%%%%%%%%%%%%%%%%%%%%%%%%%%%%%%%%%%%%%%%%%%%%%%%%%%%%%%%%%%

We shall use the above lemma to calculate $|J_r(v)|$ and $\J(v)$ when $(2^v,k)=1$.
First, we note that if $v\ge1$, then $k$ must be odd and
\eq{J_r(v) alt}{
|J_r(v)| = \begin{cases}
			2\cdot  \#\{    j\mod{ 2^{2v+1} } : j^2\equiv k+4^{v-1} r\mod{2^{2v+1}} \} &\text{if}\ 2|m,\\
						0 	&\text{if}\ 2\nmid m.
			\end{cases}
}
Indeed, when $v\ge1$, the relation $(j-mk)^2\equiv 4k+4^vr\mod{2^{2v+3}}$ implies that $2|(j-mk)$. Since $k$ is odd and we also have that $jm\equiv 0\mod{2}$, we deduce that $2\mid (m,j)$.
Hence, $|J_r(v)|=0$ when $2\nmid m$.
Assuming that $2\mid m$, we write $j=mk+2j'$ and find that
\als{
|J_r(v)| &=  \#\{j'\mod{2^{2v+2}} : j'^2\equiv k+4^{v-1} r\mod{2^{2v+1}} \}  \\
	&=  2\cdot  \#\{j\mod{2^{2v+1} } : j^2\equiv k+4^{v-1} r\mod{2^{2v+1}} \} ,
}
as claimed.

%%%%%%%%%%%%%%%%%%%%%%%%%%%%%%%%%%%%%%%%%%%%%%%%%%%%%%%%%%%%%%%%%%%%%%%%%%%%%%%%%%%%%%%%%%%%%%%%%%%%%%%%%%%%%%%%%%%%%%%%%%%%%%%%%%%%%%%%%%%%%%%%%%%%%%%%%%%%%%%%%%%%%%%%%%%%%%%%%%%%%%%%%%

\begin{lma}\label{prime 2 lma1} Let $v\ge0$ with $(2^v,k)=1$. If $m$ is odd, then
\[
\J(v) =
	\begin{cases}
		1			&\mbox{if $v=0$ and $2|k$},\\
		\frac{2}{3}		&\mbox{if $v=0$ and $2\nmid k$},\\
		0			&\text{if $v\ge1$ and $2\nmid k$}.
	\end{cases}
\]
\end{lma}

%%%%%%%%%%%%%%%%%%%%%%%%%%%%%%%%%%%%%%%%%%%%%%%%%%%%%%%%%%%%%%%%%%%%%%%%%%%%%%%%%%%%%%%%%%%%%%%%%%%%%%%%%%%%%%%%%%%%%%%%%%%%%%%%%%%%%%%%%%%%%%%%%%%%%%%%%%%%%%%%%%%%%%%%%%%%%%%%%%%%%%%

\begin{proof} The case $v\ge1$ follows by \eqref{J_r(v) alt}. Assume now that $v=0$. Since $m$ is odd, the condition $jm\equiv 0\pmod 2$ implies that every $j\in J_r(v)$ is even. Writing $j=2j'$, we deduce that
\[
|J_r(0)| = \#\{j'\mod 4: (2j'-mk)^2\equiv 4k+r\mod{8}\}
\]
If $k$ is odd, then we must have that $(2j'-mk)^2-4k\equiv -3\mod8$ and thus $r=5$, in which case $|J_r(0)|=4$; otherwise $|J_r(v)|=0$. So
\[
\J(0) = \frac{1}{2} \cdot  \frac{4}{2-(-1)} = \frac{2}{3} .
\]

Finally, assume that $k$ is even. Writing $z=j'-mk/2$, our task reduces to counting solutions to $4z^2\equiv r\mod 8$ with $1\le z\le 4$. If $r\in\{1,5\}$, then there are no such solutions, whereas if $r\in\{0,4\}$, then there are precisely two such solutions. Consequently, when $m$ is odd and $k$ is even,
\[
\J(0)=\frac{1}{2} \left( \frac{2}{2-0} + \frac{2}{2-0} \right) = 1 ,
\]
and the lemma follows in this case too.
\end{proof}

%%%%%%%%%%%%%%%%%%%%%%%%%%%%%%%%%%%%%%%%%%%%%%%%%%%%%%%%%%%%%%%%%%%%%%%%%%%%%%%%%%%%%%%%%%%%%%%%%%%%%%%%%%%%%%%%%%%%%%%%%%%%%%%%%%%%%%%%%%%%%%%%%%%%%%%%%%%%%%%%%%%%%%%%%%%%%%%%%%%%%%%

\begin{lma}\label{prime 2 lma2} Let $v\ge0$ with $(2^v,k)=1$, and suppose that $2|m$.
If $2|k$, then
\[
\J(0) =  \frac{3}{2} .
\]
If $k\equiv 1\pmod 8$, then
\[
\J(v) = \begin{cases}
		\frac{5}{6} 		&\mbox{if $v=0$}, \\
		 1			&\mbox{if $v=1$}, \\
		 2			&\mbox{if $v=2$}, \\
		\frac{14}{3} 	&\mbox{if $v\ge3$}.
	\end{cases}
\]
If $k\equiv 3,7\pmod 8$, then
\[
\J(v) = \begin{cases}
		\frac{5}{6} 		&\mbox{if $v=0$}, \\
		 \frac{4}{3}		&\mbox{if $v=1$}, \\
		0 			&\mbox{if $v\ge2$}.
	\end{cases}
\]
If $k\equiv 5\pmod 8$, then
\[
\J(v) = \begin{cases}
		\frac{5}{6} 		&\mbox{if $v=0$}, \\
		 1			&\mbox{if $v=1$}, \\
		 \frac{8}{3}		&\mbox{if $v=2$}, \\
		0 			&\mbox{if $v\ge3$}.
	\end{cases}
\]
\end{lma}

%%%%%%%%%%%%%%%%%%%%%%%%%%%%%%%%%%%%%%%%%%%%%%%%%%%%%%%%%%%%%%%%%%%%%%%%%%%%%%%%%%%%%%%%%%%%%%%%%%%%%%%%%%%%%%%%%%%%%%%%%%%%%%%%%%%%%%%%%%%%%%%%%%%%%%%%%%%%%%%%%%%%%%%%%%%%%%%%%%%%%%%

\begin{proof}
First, we calculate $|J_r(0)|$. Note that the condition $jm\equiv 0\mod2$ is trivially satisfied now since $2|m$.
Therefore, a change of variable and Lemma \ref{generic quad lemma - prime 2} imply that
\eq{J_r(0)}{
|J_r(0)| = \#\{j\mod 8: j^2\equiv 4k+r\mod{8}\} =
	\begin{cases}
		2 &\text{if}\ 4k+r\equiv 0,4\mod 8,\\
		4 &\text{if}\ 4k+r\equiv 1\mod 8,\\
		0 &\text{if}\ 4k+r\equiv 5\mod 8 .
	\end{cases}
}
Thus,
\[
\J(0) = \begin{cases}
		\frac{1}{4}\left( \frac{2}{2-0}+\frac{4}{2-1} + \frac{2}{2-0} + \frac{0}{2-(-1)} \right) = \frac{3}{2}
			&\text{if}\ 2|k,\\
		\frac{1}{4} \left(\frac{2}{2-0}+\frac{0}{2-1} + \frac{2}{2-0}+ \frac{4}{2-(-1)}\right) =  \frac{5}{6}
			&\text{if}\ 2\nmid k .
	\end{cases}
\]

Next assume that $v\ge1$, and note that the condition $(2^v,k)=1$ means that we only need consider this case when $k$ is odd.
By relation \eqref{J_r(v) alt}, we have that
\[
|J_r(v)| = 2\cdot \#\{ j\mod{2^{2v+1}}: j^2\equiv k+4^{v-1} r\mod{2^{2v+1}} \}.
\]
Now if $v\ge2$, then Lemma~\ref{generic quad lemma - prime 2} implies that $|J_r(v)|=2\cdot 4=8$ or $|J_r(v)|=0$
according to whether $k+4^{v-1}r\equiv 1\mod{8}$ or not.
Therefore, when $v\ge2$,
\[
\J(v) = \begin{cases}
		\frac{1}{4} \left(\frac{8}{2-0}+\frac{8}{2-0}\right) =  2
			&\text{if}\ v=2\ \text{and}\ k\equiv 1\mod 8,\\
		\frac{1}{4} \left( \frac{8}{2-1} + \frac{8}{2-(-1)}\right) =  \frac{8}{3}
			&\text{if}\ v=2\ \text{and}\ k\equiv 5\mod 8,\\
		\frac{1}{4} \left(\frac{8}{2-0}+\frac{8}{2-1} + \frac{8}{2-0}+ \frac{8}{2-(-1)}\right) =  \frac{14}{3}
			&\text{if}\ v\ge3\ \text{and}\ k\equiv 1\mod 8,\\
		0 &\text{otherwise} .
	\end{cases}
\]
Finally, we consider the case $v=1$.
Using  Lemma~\ref{generic quad lemma - prime 2} again, we have
\[
|J_r(1)| = 2\cdot \#\{ j\mod{8}: j^2\equiv k+r\mod 8 \}
	= \begin{cases}
		4 &\text{if}\ k+r\equiv 0,4\mod 8,\\
		8 &\text{if}\ k+r\equiv 1\mod 8,\\
		0 &\text{otherwise}.
	\end{cases}
\]
Therefore,
\[
\J(1) = \begin{cases}
		\frac{1}{4} \cdot \frac{8}{2-0} = 1
			&\text{if}\ k\equiv 1,5\mod 8,\\
		\frac{1}{4} \left( \frac{4}{2-1}  + \frac{4}{2-(-1)} \right) 	= \frac{4}{3}
			&\text{if}\ k\equiv 3,7 \mod 8 ,
	\end{cases}
\]
which completes the proof of the lemma.
\end{proof}

%%%%%%%%%%%%%%%%%%%%%%%%%%%%%%%%%%%%%%%%%%%%%%%%%%%%%%%%%%%%%%%%%%%%%%%%%%%%%%%%%%%%%%%%%%%%%%%%%%%%%%%%%%%%%%%%%%%%%%%%%%%%%%%%%%%%%%%%%%%%%%%%%%%%%%%%%%%%%%%%%%%%%%%%%%%%%%%%%%%%%%%

Lemma \ref{formula for J} now follows as a direct consequence of Lemmas \ref{prime 2 lma1} and \ref{prime 2 lma2}.

%%%%%%%%%%%%%%%%%%%%%%%%%%%%%%%%%%%%%%%%%%%%%%%%%%%%%%%%%%%%%%%%%%%%%%%%%%%%%%%%%%%%%%%%%%%%%%%%%%%%%%%%%%%%%%%%%%%%%%%%%%%%%%%%%%%%%%%%%%%%%%%%%%%%%%%%%%%%%%%%%%%%%%%%%%%%%%%%%%%%%%%%%%%%%%%%%%%%%%%%%%%%%%%%%%%%%%%%%%%%%%%%%%%%%%%%%%%%%%%%%%%%%%%%%%%%%%%%%%%%%%%%%%%%%%%%%%%%%%%%%%%%%%%%%%%%%%%%%%%%%%%%%%%%%%%%%%%%%%%%%%%%%%%%%%%%%%%%%%%%%%%%%%%%%%%%%%%%%%%%%%%%%%%%%%%%%%%%%%%%%%%%%%%%%%%%%%%%%%%%%%%%%%%%%%%%%%%%%%%%%%%%%%%%%%%%%%%%%%%%%%%%%%%%%%%%%%%%%%%%%%%%%%%%%%%%%%%%%%%%%%%%%%%%%%%%%%%%%%%%%%%%%%%%%%%%%%%%%%%%%%%%%%%%%%%%%%%%%%%%%%%%%%%%%%%%%%%%%%

\appendix

\section{by Chantal David, Greg Martin and Ethan Smith}

The purpose of this appendix is to give a probabilistic interpretation to the Euler factors arising in $K(G)\frac{|G|}{|\Aut(G)|}$ and $K(N)\frac{N}{\phi(N)}$,
where $K(G)$ and $K(N)$ are defined by~\eqref{define K(G)} and~\eqref{define K(N)}, respectively.
Given a prime $\ell$, we let $\nu_\ell(\cdot)$ denote the usual $\ell$-adic valuation.
For each integer $e\ge 1$, we also let $\GL_2(\Z/\ell^e\Z)$ denote the usual group of invertible $2\times 2$ matrices with entries from $\Z/\ell^e\Z$.
The $2\times 2$ identity matrix we denote by $I$.
The main results of this appendix are as follows.

%%%%%%%%%%%%%%%%%%%%%%%%%%%%%%%%%%%%%%%%%%%%%%%%

\begin{thm}\label{K(N) interpretation}
For each positive integer $N$,
\begin{equation*}
\frac{K(N)\cdot N}{\phi(N)}
=\prod_\ell\left(\lim_{e\rightarrow\infty}\frac{\ell^e\cdot\#\{\sigma\in\GL_2(\Z/\ell^e\Z) : \det(\sigma)+1-\tr(\sigma)\equiv N\pmod{\ell^e}\}}{\#\GL_2(\Z/\ell^e\Z)}\right),
\end{equation*}
where the product is taken over all primes $\ell$.
Furthermore, the sequences defining the Euler factors are constant for $e>\nu_\ell(N)$.
\end{thm}

\begin{rmk}
If $\mu$ denotes the Haar measure on the space of $2\times 2$ matrices over the $\ell$-adic integers $\Z_\ell$, normalized so that $\mu\left(\GL_2(\Z_\ell)\right)=1$,
then the Euler factor of $K(N)\frac{N}{\phi(N)}$ for the prime $\ell$ may be viewed as the density function for the probability measure on $\Z_\ell$
defined by the pushforward of $\mu$ via the map $\det+1-\tr:\GL_2(\Z_\ell)\rightarrow\Z_\ell$.
\end{rmk}

%%%%%%%%%%%%%%%%%%%%%%%%%%%%%%%%%%%%%%%%%%%%%%%%

\begin{thm}\label{K(G) interpretation}
For each pair of positive integers $m$ and $k$, put $G=G_{m,k}=\Z/m\Z\times\Z/mk\Z$.
Then
\begin{equation*}
\frac{K(G)\cdot |G|}{|\Aut(G)|}
=\prod_\ell\left(\lim_{e\rightarrow\infty}\frac{\ell^e\cdot\#\left\{\sigma\in\GL_2(\Z/\ell^e\Z) :
	\begin{array}{l}\det(\sigma)+1-\tr(\sigma)\equiv |G|\pmod{\ell^e},\\
	\sigma\equiv I\pmod{\ell^{\nu_\ell(m)}},\\
	\sigma\not\equiv I\pmod{\ell^{\nu_\ell(m)+1}}
	\end{array}\right\}}
	{\#\GL_2(\Z/\ell^e\Z)}\right),
\end{equation*}
where the product is taken over all primes $\ell$.
Furthermore, the sequences defining the Euler factors are constant for $e>\nu_\ell(|G|)$.
\end{thm}

%%%%%%%%%%%%%%%%%%%%%%%%%%%%%%%%%%%%%%%%%%%%%%%%

For the remainder of this appendix, we assume that $e, n, N,$ and $\ell$ are positive integers with $\ell$ prime and $n^2\mid N$.
Later we will also assume that $N=|G|=m^2k$.
For convenience, we let
\begin{equation*}
C_{N,n}(\ell^e)=\left\{\sigma\in\GL_2(\Z/\ell^e\Z) :
	\det(\sigma)+1-\tr(\sigma)\equiv N\pmod{\ell^e},\ \sigma\equiv I\pmod{\ell^{\nu_\ell(n)}}\right\}.
\end{equation*}
In the case that $\ell\nmid n$, we note that the condition $\sigma\equiv I\pmod{\ell^{\nu_\ell(n)}}$ is vacuous.
As usual, $\leg{\cdot}{\ell}$ denotes the Kronecker symbol modulo $\ell$.

%%%%%%%%%%%%%%%%%%%%%%%%%%%%%%%%%%%%%%%%%%%%%%%%%%%%%%%%%%%%

\begin{lma}\label{matrix count mod ell}
If $\ell\nmid n$, then
\begin{equation*}
\#C_{N,n}(\ell)=
\ell\left(\ell^2-\leg{N}{\ell}^2\ell-1-\leg{N-1}{\ell}^2\right).
\end{equation*}
\end{lma}

\begin{proof}
We first observe that $\#C_{N,n}(\ell)$ is equal to the number of quadruples $(a,b,c,d)$ satisfying $0\le a,b,c,d<\ell$ and
\begin{align}
ad-bc+1-(a+d)&\equiv N\pmod\ell,\label{det-tr cond}\\
ad-bc&\not\equiv 0\pmod\ell\label{det cond}.
\end{align}
The lemma follows by first counting the number of quadruples satisfying~\eqref{det-tr cond}
and then removing the number of quadruples satisfying~\eqref{det-tr cond} that do not satisfy~\eqref{det cond}.

Rearranging, we see that the condition~\eqref{det-tr cond} may be rewritten as
\begin{equation*}
(a-1)(d-1)-bc\equiv N\pmod\ell.
\end{equation*}
It is clear that any choice of $a,b,c$ with $a\ne 1$ uniquely determines $d$.
On the other hand, if $a=1$, then there are $\ell$ choices for $d$, and the pair $(b,c)$ must satisfy $bc\equiv -N\pmod\ell$.
Therefore, there are
\begin{equation*}
\ell^3+\left(1-\leg{N}{\ell}^2\right)\ell^2-\ell
\end{equation*}
solutions $(a,b,c,d)$ to~\eqref{det-tr cond} with $0\le a,b,c,d<\ell$.

We now count the number of quadruples $(a,b,c,d)$ with $0\le a,b,c,d<\ell$
for which~\eqref{det-tr cond} holds but~\eqref{det cond} does not.
These are the quadruples that satisfy the system
\begin{align*}
a+d&\equiv 1-N\pmod\ell,\\
ad&\equiv bc\pmod\ell.
\end{align*}
It is clear that any choice of $a$ uniquely determines $d$.
If $a=0$ or $a=1-N$, then there are $2\ell-1$ choices for the pair $(b,c)$.
On the other hand, if $a\ne 0,1-N$, there are only $\ell-1$ choices for $(b,c)$.
Therefore, there are
\begin{equation*}
\ell^2+\leg{N-1}{\ell}^2\ell
\end{equation*}
solutions $(a,b,c,d)$ to~\eqref{det-tr cond} with $0\le a,b,c,d<\ell$ for which~\eqref{det cond} does not hold.
\end{proof}

%%%%%%%%%%%%%%%%%%%%%%%%%%%%%%%%%%%%%%%%%%%%%%%%%%%%%%%%%%%%

\begin{prop}\label{matrix proportions for ell not dividing N}
If $\ell\nmid N$, then
\begin{equation*}
\#C_{N,n}(\ell^e)=\ell^{3(e-1)+1}\left(\ell^2-\ell-1-\leg{N-1}{\ell}^2\right)
\end{equation*}
for every $e\ge 1$.
\end{prop}

\begin{proof}
The case $e=1$ is treated in Lemma~\ref{matrix count mod ell}, and so we assume that $e\ge 2$.
Since any $\sigma\in C_{N,n}(\ell^e)$ must reduce modulo $\ell$ to a matrix in $C_{N,n}(\ell)$,
it suffices to count the number of matrices in $C_{N,n}(\ell^e)$ that reduce to a given matrix in $C_{N,n}(\ell)$.
To this end, we assume that $\sigma_0\in C_{N,n}(\ell)$ and $\sigma\in C_{N,n}(\ell^e)$ is such that $\sigma\equiv\sigma_0\pmod\ell$.
Thus, we may write
\begin{equation*}
\sigma_0=\begin{pmatrix}a_0&b_0\\ c_0&d_0\end{pmatrix}\quad\text{and}\quad
\sigma=\begin{pmatrix}a_0+a\ell&b_0+b\ell\\ c_0+c\ell&d_0+d\ell\end{pmatrix}
\end{equation*}
with $0\le a_0,b_0,c_0,d_0<\ell$ and $0\le a,b,c,d<\ell^{e-1}$.
Note that the condition $\det\sigma\not\equiv 0\pmod\ell$ is necessarily satisfied since $\det\sigma\equiv\det\sigma_0\pmod\ell$ and
$\sigma_0\in C_{N,n}(\ell)$.
Therefore, $\sigma\in C_{N,n}(\ell^e)$ if and only if
\begin{equation}\label{lift det-tr cond}
a_0d_0-b_0c_0+1-a_0-d_0
+(a(d_0-1)+d(a_0-1)-b_0c-bc_0)\ell
+(ad-bc)\ell^2
\equiv N\pmod{\ell^e}.
\end{equation}
Since $\sigma_0\in C_{N,n}(\ell)$, it follows that $a_0d_0-b_0c_0+1-a_0-d_0=N+k_0\ell$ for some $k_0$,
and hence condition~\eqref{lift det-tr cond} reduces to
\begin{equation*}
k_0+
((d_0-1)a-c_0b-b_0c+(a_0-1)d)
+(ad-bc)\ell
\equiv 0\pmod{\ell^{e-1}}.
\end{equation*}
Since $\ell\nmid N$, $\sigma_0$ cannot be the identity matrix modulo $\ell$, 
and the polynomial $(d_0-1)a - c_0 b - b_0 c + (a_0-1) d$ in the variables $a,b,c,d$ has at least one nonzero coefficient. 
Say for example that $d_0 - 1$ is not zero.
Then for each triple $(b,c,d)$, there is a unique choice of $a$ satisfying the above congruence.
Therefore, there are exactly $\ell^{3(e-1)}$ solutions $(a,b,c,d)$ with $0\le a,b,c,d<\ell^{e-1}$.
\end{proof}

%%%%%%%%%%%%%%%%%%%%%%%%%%%%%%%%%%%%%%%%%%%%%%%%%%%%%%%%%%%%

Let $\Mat_2(\Z/\ell^k\Z)$ denote the ring of $2\times 2$ matrices with entries from $\Z/\ell^k\Z$.
In order to compute $C_{N,n}(\ell^e)$ when $\ell\mid N$ we need to know the number of matrices in $\Mat_2(\Z/\ell^k\Z)$ of every individual determinant.

\begin{prop}\label{det prop}
Let $M$ be a positive integer, and let $r=\nu_\ell(M)$.  Then for $r,s\ge 0$, we have
\begin{equation*}
\#\left\{\sigma\in\Mat_2(\Z/\ell^{r+s}\Z) : \det(\sigma)\equiv M\pmod{\ell^{r+s}}\right\}
=\ell^{2(r-1)}\left(\ell^{3s}(\ell+1)(\ell^{r+1}-1)+\delta(s)\right),
\end{equation*}
where $\delta(s)$ is defined by
\begin{equation*}
\delta(s):=\begin{cases}
1&\text{if }s=0,\\
0&\text{otherwise}.
\end{cases}
\end{equation*}
\end{prop}

For the proof of Proposition~\ref{det prop}, we first make a simple reduction and fix some notation.
Given any positive integer $M$, we write $M=\ell^rM'$ with $r=\nu_\ell(M)$ and $(M',\ell)=1$.
Since the determinant maps $\GL_2(\Z/\ell^{r+s}\Z)$ onto $(\Z/\ell^{r+s}\Z)^*$,
it follows that there is an $\alpha\in\GL_2(\Z/\ell^{r+s}\Z)$ such that $\det(\alpha)\equiv M'\pmod{\ell^{r+s}}$.
Since the map $\sigma\mapsto\alpha\sigma$ is a group automorphism of $\Mat_2(\Z/\ell^{r+s}\Z)$
and since $\det(\sigma)=M=\ell^rM'$ if and only if $\det(\alpha^{-1}\sigma)=\ell^r$,
it follows that
\begin{equation*}
\#\left\{\sigma\in\Mat_2(\Z/\ell^{r+s}\Z) : \det(\sigma)\equiv M\pmod{\ell^{r+s}}\right\}
=\#F(r,s),
\end{equation*}
where
\begin{equation*}
F(r,s):=\left\{\sigma\in\Mat_2(\Z/\ell^{r+s}\Z) : \det(\sigma)\equiv \ell^r\pmod{\ell^{r+s}}\right\}.
\end{equation*}
Thus, we see that $\#\left\{\sigma\in\Mat_2(\Z/\ell^{r+s}\Z) : \det(\sigma)\equiv M\pmod{\ell^{r+s}}\right\}$ depends on the power of $\ell$ dividing $M$ and not on the $\ell$-free part of $M$.
With this in mind, we define
\begin{equation*}
f(r,s):=\#F(r,s),
\end{equation*}
where we adopt the natural convention that $f(0,0)=1$.
Proposition~\ref{det prop} then follows easily by induction on $r$ using the following lemma.

\begin{lma}
For every $s\ge 0$, we have
\begin{align*}
f(0,s)&=\ell^{3s-2}(\ell^2-1)+\ell^{-2}\delta(s),\\
f(1,s)&=\ell^{3s}(\ell+1)(\ell^2-1)+\delta(s),\\
f(r,s)&=\ell^{3(r+s-1)}(\ell+1)(\ell^2-1)+\ell^4f(r-2,s),\quad r\ge 2.
\end{align*}
\end{lma}

\begin{proof}
By convention we have $f(0,0)=1$.
For $s\ge 1$, we have the well-known formula
\begin{equation*}
f(0,s)=\#\SL_2(\Z/\ell^s\Z)=\ell^{3s-2}(\ell^2-1).
\end{equation*}
This proves the first formula given in the statement of the lemma.

Now assume that $r\ge 1$.  If $r=1$ and $s=0$, then we have
\begin{equation*}
f(1,0)=\#\Mat_2(\Z/\ell\Z)-\#\GL_2(\Z/\ell\Z)=\ell^3+\ell^2-\ell.
\end{equation*}
We observe that any $\sigma\in F(r,s)$ must reduce modulo $\ell$ to some $\sigma_0\in F(1,0)$.
Thus, we assume that $\sigma_0\in F(1,0)$, and we write
\begin{equation*}
\sigma_0=\begin{pmatrix}a_0&b_0\\ c_0&d_0\end{pmatrix}\quad\text{and}\quad
\sigma=\begin{pmatrix}a_0+a\ell&b_0+b\ell\\ c_0+c\ell&d_0+d\ell\end{pmatrix},
\end{equation*}
with $0\le a_0, b_0, c_0, d_0<\ell$ and $0\le a,b,c,d<\ell^{r+s-1}$.
By definition, we see that $\sigma\in F(r,s)$ if and only if
\begin{equation*}
a_0d_0-b_0c_0+(d_0a-c_0b-b_0c+a_0d)\ell+(ad-bc)\ell^2\equiv\ell^r\pmod{\ell^{r+s}}.
\end{equation*}
If $\sigma_0$ is not the zero matrix modulo $\ell$, then there are exactly $\ell^{3(r+s-1)}$ choices of $(a,b,c,d)$ satisfying the above congruence.
On the other hand, if $\sigma_0$ is the zero matrix (which is always an element of $F(1,0)$), the above congruence condition reduces to
\begin{equation}\label{zero matrix cond}
(ad-bc)\ell^2\equiv \ell^r\pmod{\ell^{r+s}}.
\end{equation}
If $r=1$, then there can be no solutions to~\eqref{zero matrix cond} with $s\ge 1$.
Therefore,
\begin{equation*}
f(1,s)=\ell^{3s}(f(1,0)-1)=\ell^{3s}(\ell^3+\ell^2-\ell-1)
=\ell^{3s}(\ell+1)(\ell^2-1)
\end{equation*}
when $s\ge 1$, and this completes the proof of the second formula stated in the lemma.
On the other hand, if $r\ge 2$, then condition~\eqref{zero matrix cond} reduces to
\begin{equation*}
(ad-bc)\equiv\ell^{r-2}\pmod{\ell^{r-2+s}}.
\end{equation*}
There are $\ell^4f(r-2,s)$ solutions to this congruence with $0 \leq a,b,c,d < \ell^{r+s-1}$.
Whence
\begin{equation*}
\begin{split}
f(r,s)&=\ell^{3(r+s-1)}(f(1,0)-1)+\ell^4f(r-2,s)\\
&=\ell^{3(r+s-1)}(\ell+1)(\ell^2-1)+\ell^4f(r-2,s)
\end{split}
\end{equation*}
for $r\ge 2$, and this completes the proof of the lemma.
\end{proof}

%%%%%%%%%%%%%%%%%%%%%%%%%%%%%%%%%%%%%%%%%%%%%%%%%%%%%%%%%%%%

\begin{prop}\label{ell divides N but not n}
If $v=\nu_\ell(N)\ge 1$ and $\ell\nmid n$, then
\begin{equation*}
\#C_{N,n}(\ell^e)
=\ell^{3e-v-2}(\ell+1)\left(\ell^{v+1}-\ell^v-1\right)
\end{equation*}
for every $e>v$.
\end{prop}

\begin{proof}
By Lemma~\ref{matrix count mod ell}, we have
\begin{equation}\label{base case}
\#C_{N,n}(\ell)=\ell(\ell^2-2)=\ell^3-2\ell,
\end{equation}
and so we may assume that $e\ge 2$.
We proceed in a manner similar to the proof of Proposition~\ref{matrix proportions for ell not dividing N}.
In particular, we assume that $\sigma_0\in C_{N,n}(\ell)$ and count the number of $\sigma\in C_{N,n}(\ell^e)$ that reduce to $C_{N,n}(\ell)$.
Writing
\begin{equation*}
\sigma_0=\begin{pmatrix}a_0&b_0\\ c_0&d_0\end{pmatrix}\quad\text{and}\quad
\sigma=\begin{pmatrix}a_0+a\ell&b_0+b\ell\\ c_0+c\ell&d_0+d\ell\end{pmatrix}
\end{equation*}
with $0\le a_0,b_0,c_0,d_0<\ell$ and $0\le a,b,c,d<\ell^{e-1}$, we deduce that the quadruple $(a,b,c,d)$ must
satisfy~\eqref{lift det-tr cond}.
As in the proof of Proposition~\ref{matrix proportions for ell not dividing N}, if $\sigma_0$ is not the identity matrix,
there are exactly $\ell^{3(e-1)}$ choices for $(a,b,c,d)$.

Now suppose that $\sigma_0$ is the identity matrix.
(Note that the identity matrix is always an element of $C_{N,n}(\ell)$ when $\ell\mid N$.)
Then writing $N=\ell^vN'$ with $v=\nu_\ell(N)\ge 1$ and $(N',\ell)=1$,
we see that condition~\eqref{lift det-tr cond} reduces to
\begin{equation}\label{id case}
(ad-bc)\ell^2\equiv N'\ell^{v}\pmod{\ell^{e}}.
\end{equation}
Clearly, there are no solutions to this congruence unless $v\ge 2$.
Therefore, if $v=1$ and $e\ge 2$, we have that
\begin{equation*}
\#C_{N,n}(\ell^e)
=\ell^{3(e-1)}(\ell^3-2\ell-1)
=\ell^{3e-3}(\ell+1)(\ell^2-\ell-1).
\end{equation*}
Now, suppose that $v\ge 2$ and $e\ge 3$.  Then~\eqref{id case} reduces to
\begin{equation}
(ad-bc)\equiv N'\ell^{v-2}\pmod{\ell^{e-2}}.
\end{equation}
The number of solutions to this congruence with $0\le a,b,c,d<\ell^{e-1}$ is equal to
\begin{equation*}
\ell^4\#\{\alpha\in\Mat_2(\Z/\ell^{e-2}\Z): \det(\alpha)\equiv N'\ell^{v-2}\pmod{\ell^{e-2}}\}.
\end{equation*}
Since we are assuming that $v<e$, Proposition~\ref{det prop} implies that the above count is equal to
\begin{equation*}
\ell^4\ell^{2(v-3)}\ell^{3(e-v)}(\ell+1)(\ell^{v-1}-1)
=\ell^{3e-v-2}(\ell+1)(\ell^{v-1}-1).
\end{equation*}
Putting everything together, we find that
\begin{equation*}
\begin{split}
\#C_{N,n}(\ell^e)
&=\ell^{3(e-1)}(\ell^3-2\ell-1)+\ell^{3e-v-2}(\ell+1)(\ell^{v-1}-1)\\
&=\ell^{3e-v-2}(\ell+1)\left(\ell^{v+1}-\ell^v-1\right)
\end{split}
\end{equation*}
for $v\ge 2$.
\end{proof}

%%%%%%%%%%%%%%%%%%%%%%%%%%%%%%%%%%%%%%%%%%%%%%%%%%%%%%%%%%%%

Recall our standing assumption that $n^2\mid N$.
\begin{thm}\label{complete matrix count thm}
Let $u=\nu_\ell(n)$ and $v=\nu_\ell(N)$.
Then for every $e>v$, we have
\begin{equation*}
\#C_{N,n}(\ell^e)
=\begin{cases}
\ell^{3(e-1)+1}\left(\ell^2-\ell-1-\leg{N-1}{\ell}^2\right)&\text{if }u=0\text{ and }v= 0,\\
\ell^{3e-v-2}(\ell+1)\left(\ell^{v+1}-\ell^v-1\right)&\text{if }u=0\text{ and }v\ge 1,\\
\ell^{3e-v-2}(\ell+1)(\ell^{v-2u+1}-1)&\text{if }1\le u\le v/2,\\
0&\text{if } 0\le v/2<u.
\end{cases}
\end{equation*}
Therefore, for every $e>v$, we have
\begin{equation*}
\begin{split}
\frac{\ell^e\#C_{N,n}(\ell^e)}{\#\GL_2(\Z/\ell^e\Z)}
&=\begin{cases}
\displaystyle
\left(1-\frac{\leg{N-1}{\ell}^2\ell+1}{(\ell-1)^2(\ell+1)}\right)&\text{if }u=0\text{ and }v= 0,\\
\displaystyle
\frac{\ell}{\ell-1}\left(1-\frac{1}{\ell^v(\ell-1)}\right)&\text{if }u=0\text{ and }v\ge 1,\\
\displaystyle
\frac{\ell}{\ell^{2u}(\ell-1)}\left(\frac{\ell^{v+1}-\ell^{2u}}{\ell^{v+1}-\ell^v-1}\right)\left(1-\frac{1}{\ell^v(\ell-1)}\right)&\text{if }1\le u\le v/2,\\
0&\text{if }0\le v/2<u.
\end{cases}
\end{split}
\end{equation*}
\end{thm}

\begin{proof}
Note that the second assertion of theorem follows from the first together with the well-known formula
\begin{equation*}
\#\GL_2(\Z/\ell^e\Z)=\ell^{4(e-1)+1}(\ell+1)(\ell-1)^2,
\end{equation*}
and so it suffices to prove the first assertion of the theorem.

The first two cases have already been addressed by Propositions~\ref{matrix proportions for ell not dividing N} and~\ref{ell divides N but not n}.
Therefore, we may assume that $u\ge 1$.
Supposing that $\sigma\in C_{N,n}(\ell^e)$, we may write
\begin{equation*}
\sigma=\begin{pmatrix}1+a\ell^u&b\ell^u\\ c\ell^u&1+d\ell^u\end{pmatrix}
\end{equation*}
with $0\le a,b,c,d<\ell^{e-u}$ chosen such that
\begin{equation*}
(ad-bc)\ell^{2u}\equiv N'\ell^v\pmod{\ell^e}.
\end{equation*}
This congruence clearly has no solutions if $e>v$ and $2u>v$.
Therefore, we may assume that $2\le 2u\le v<e$.
In this case the above congruence is equivalent to the condition
\begin{equation*}
(ad-bc)\equiv N'\ell^{v-2u}\pmod{\ell^{e-2u}}
\end{equation*}
for $0 \leq a,b,c,d < \ell^{e-u}$.
Applying Proposition~\ref{det prop} with $r=v-2u$ and $s=e-v>0$, we find that
\begin{equation*}
\begin{split}
\#C_{N,n}(\ell^e)
&=\ell^{4u}\ell^{2(v-2u-1)}\ell^{3(e-v)}(\ell+1)(\ell^{v-2u+1}-1)\\
&=\ell^{3e-v-2}(\ell+1)(\ell^{v-2u+1}-1).
\end{split}
\end{equation*}
\end{proof}

%%%%%%%%%%%%%%%%%%%%%%%%%%%%%%%%%%%%%%%%%%%%%%%%%%%%%%%%%%%%

We are now ready to give the proofs of Theorems~\ref{K(N) interpretation} and~\ref{K(G) interpretation}.

\begin{proof}[Proof of Theorems~\ref{K(N) interpretation} and~\ref{K(G) interpretation}]
Theorem~\ref{K(N) interpretation} follows easily from~\eqref{define K(N)} and the cases of Theorem~\ref{complete matrix count thm} with $\nu_\ell(n)=u=0$.
For the proof of Theorem~\ref{K(G) interpretation}, we let $N=m^2k=|G|$, and for each prime $\ell$, we put
\begin{equation*}
v_\ell(N,n):=\frac{\ell^e\#C_{N,n}(\ell^e)}{\#\GL_2(\Z/\ell^e\Z)}
\end{equation*}
with $e=e_\ell>\nu_\ell(N)$.
We then compute the absolutely convergent infinite product
\begin{equation*}
\prod_\ell\left(v_\ell(N,m)-v_\ell(N,\ell m)\right)
\end{equation*}
in two different ways.
On the one hand, by definition of the $v_\ell(N,n)$, the above expression is equal to
\begin{equation*}
\prod_\ell\left(\frac{\ell^e\cdot\#\left\{\sigma\in\GL_2(\Z/\ell^e\Z) :
	\begin{array}{l}\det(\sigma)+1-\tr(\sigma)\equiv N\pmod{\ell^e},\\
	\sigma\equiv I\pmod{\ell^{\nu_\ell(m)}},\\
	\sigma\not\equiv I\pmod{\ell^{\nu_\ell(m)+1}}
	\end{array}\right\}}
	{\#\GL_2(\Z/\ell^e\Z)}\right).
\end{equation*}
On the other hand, by comparing~\eqref{define K(G)} and Lemma~\ref{Aut(G)} with Theorem~\ref{complete matrix count thm},
we see that it is equal to $K(G)\frac{|G|}{|\Aut(G)|}$.
\end{proof}

\def\cprime{$'$}
\bibliographystyle{alpha}

%\bibliographystyle{alpha}
%\bibliography{references}

\end{document}